\documentclass[leqno,11pt]{amsart}
\usepackage{amsmath,amssymb,amsfonts,amscd}
\usepackage[mathscr]{eucal}
\usepackage{tikz}
\usepackage{verbatim}
\usepackage[colorlinks,pagebackref]{hyperref}
\usepackage{rotating}
\usepackage{color}

\theoremstyle{plain}
\newtheorem{thm}{Theorem}[section]
\newtheorem{lem}[thm]{Lemma}

\newtheorem{prop}[thm]{Proposition}

\newtheorem*{thm*}{Theorem}

\newtheorem{sublem}[equation]{Lemma}
\newtheorem{subcor}[equation]{Corollary}
\newtheorem{subprop}[equation]{Proposition}

\theoremstyle{definition}
\newtheorem{defn}[thm]{Definition}

\newtheorem{cosa}[thm]{}

\theoremstyle{remark}

\newtheorem{subrem}[equation]{Remark}
\newtheorem{subrems}[equation]{Remarks}
\newtheorem{subrems*}{Remarks.}

\numberwithin{equation}{thm}

\newcommand{\D}{\boldsymbol{\mathsf{D}}}

\newcommand{\LL}{\mathsf L}
\newcommand{\R}{\mathsf R}

\newcommand{\sE}{\mathscr{E}}

\newcommand{\1}{\mathbf{1}}

 \newcommand{\Rf}{\R f^{}_{\<\<*}}

\newcommand{\fst}{{f^{}_{\<\<*}}}

\newcommand{\ush}[1]{{#1^{\textup{\texttt\#}}}}

\newcommand{\qc}{\mathsf{qc}}

\newcommand{\Dqc}{\D_{\mathsf{qc}}}

\newcommand\Dqcpl{\D_\qc^{\lift.95,\text{\cmt\char'053},}}

\font\cmt=cmtex10

\newcommand{\CH}{\mathcal H}
\newcommand{\CI}{\mathcal I}
\newcommand{\CO}{\mathcal O}

\newcommand{\CT}{\mathcal T}

\newcommand{\bpic}{\begin{tikzpicture}}
\newcommand{\epic}{\end{tikzpicture}}

\newcommand{\Otimes}[1]{\otimes^\LL_{#1}}

\newcommand{\sHom}{\CH om}

\newcommand{\Hqc}[1]{\sHom_{#1}^{\<\qc}}

\newcommand{\set}{\!:=}

\newcommand{\sX}{{\<\<X}}

\newcommand{\sst}{\scriptstyle}
\newcommand{\sss}{\scriptscriptstyle}

\newcommand{\smallcirc}{{\>\>\lift1,\sst{\circ},\,}}

\newcommand{\<}{\mkern-1mu}
\renewcommand{\>}{\mkern1mu}
\newcommand{\va}[1]{\vspace{#1pt}}

\newcommand{\kf}{\kern.5pt}

\def\lift#1,#2,{\vbox to 0pt{\vskip-#1 ex\hbox{$\scriptstyle #2$}\vss}}

\newcommand{\OX}{\mathcal O_{\<\<X}}
\newcommand{\OY}{\mathcal O_Y}
\newcommand{\OZ}{\mathcal O_{\<Z}}

\newcommand{\OW}{\mathcal O_W}

\newcommand{\OU}{\mathcal O_U}

\newcommand{\fc}[1]{{\mathsf{C}}^{}_{#1}}
 
\newcommand{\fundamentalclassa}[1]{{\boldsymbol{\mathsf{a}}}_{#1}}

\newcommand{\fundamentalclassb}[1]{{\boldsymbol{\mathsf{b}}}_{#1}}

\newcommand{\circled}[1]{\textcircled{\scriptsize{#1}}}

\newcommand{\lto}{\longrightarrow}
\newcommand{\xto}{\xrightarrow}

\newcommand\iso{{\mkern8mu\longrightarrow \mkern-25.5mu{}^\sim\mkern17mu}}
\newcommand\osi{{\mkern8mu\longleftarrow \mkern-25.5mu{}^{\>\sim}\mkern17mu}}

\DeclareMathOperator{\spec}{Spec}

\DeclareMathOperator{\stor}{\CT\!\mathit{or}}
\DeclareMathOperator{\id}{id}

\DeclareMathOperator{\via}{{\textup{via}}}

\DeclareMathOperator{\ev}{ev}

\def\Iso{\vbox to 0pt{\vss\hbox{$\widetilde{\phantom{nn}}$}\vskip-7pt}}

\def\cA #1; #2;{\cite[p.\,#1, #2]{A}}
\def\cT #1; #2;{\cite[p.\,#1, #2]{T}}

\def\lift#1,#2,{\vbox to 0pt{\vskip-#1 ex\hbox{$\scriptstyle #2$}\vss}}

\def\drlm#1{\underset{\vtop{\vskip-4.2pt\hbox to 14pt{\rightarrowfill} \vskip-10pt\hbox{$\scriptstyle \ #1$}}}\to\lim\,}

\def\dirlm#1{\lim\hskip-1.65em\lower1.37ex
       \hbox{\smash[b]{$
                   \underset{\lift 1.37,
                                         {\hbox to 0pt{\hss$\scriptscriptstyle#1$\hss}},
                                  }
                     {\:\hbox to 1.37em {\rightarrowfill}}
               $} }                      
     \!\<}

\begin{document}

\title[Fundamental class of an essentially smooth scheme\kf-map]{On the fundamental class of an\\
essentially smooth scheme\kf-map}

\author[J.\,Lipman]{Joseph Lipman} 
\address{142 Ranch Ln., Santa Barbara CA 93111, U.S.A.}
\email{jlipman@purdue.edu}

\author[A.\,Neeman]{Amnon Neeman} 
\address{Centre for Mathematics and its Applications, Mathematical Sciences\linebreak  Institute,
Australian National University, Canberra, ACT 0200, Australia.}
\email{Amnon.Neeman@anu.edu.au}

\thanks{This article is based on work supported by the National Science Foundation under Grant No. 0932078000, while the authors were in residence,  during the Spring semester of 2013, at the Mathematical Sciences Research Institute in Berkeley, California.}
  
\keywords{Grothendieck duality, 
relative dualizing complex,
fundamental class.}

\subjclass[2010]{Primary: 14F05, 14F10. Secondary:  13D03}

%{\today}

\begin{abstract}   Let $f\colon X\to Z$ be a separated essentially-finite\kf-type flat map of noetherian schemes,  and $\delta\colon X\to X\times_{\<Z}X$ the diagonal map. 
The \emph{fundamental class} $\fc{\<f}$ (globalizing  residues)  is a 
map~from the relative Hochschild functor 
$\LL\delta^*\<\delta_{\<*}f^*$ to the relative dualizing functor~$f^!\<$.  A~compatibility between this 
$\fc{\<f}$ and derived tensor product is shown.
The main result is that, in a suitable sense,  $\fc{\<f}$ generalizes Verdier's classical isomorphism  for smooth $f\<$ with fibers of dimension~$d$,  an isomorphism 
that binds $f^!$ to relative $d$-forms. 

\end{abstract}

\maketitle

\section*{Introduction}

\stepcounter{section}
 
\begin{cosa}\label{motive} (Underlying duality theory.)
For a scheme $X\<$, $\D(X)$ is the derived category of the abelian category of $\OX$-modules;
and $\Dqc(X)\subset\D(X)$  (resp.~$\Dqcpl(X)\subset\D(X)$) is the full subcategory spanned by the complexes $C$ such that the cohomology modules $H^i(C)$ are all quasi-coherent 
(resp.~are all quasi-coherent, \emph{and} vanish for all but finitely many $i<0$).

Grothendieck duality is concerned with a pseudofunctor~$(-)^!$ over the category $\sE$ of essentially-finite\kf-type separated maps of noetherian schemes, taking values in 
$\Dqcpl(-)$. This pseudofunctor is uniquely determined up to isomorphism by the following three properties:\vspace{1pt}

(i) For formally \'etale $\sE$-maps $f$,  $f^!$ is the usual restriction functor~$f^*\<$.

(ii) (Duality) If $f$ is a proper map of noetherian schemes then $f^!$ is right-adjoint to $\Rf$.

(iii) Suppose there is given a fiber square in $\sE$
\begin{equation}\label{fiber}
\CD
X'@>v>> X\\
@VgV\mkern 20mu V @VVfV \\
Z' @>\lift4.8,\displaystyle\clubsuit,>u> Z
\endCD
\end{equation}
with $f$ (hence $g$) proper and $u$ (hence $v$) formally \'etale.
Then  the functorial \emph{base-change map}
\begin{equation}\label{beta}
\beta_{\clubsuit}(F\>)\colon:v^*\<\<f^!\<\<F\to g^!\<u^*\<\<F\qquad\big(F\in\Dqcpl(Z)\big),
\end{equation}
defined to be  adjoint to the natural composition
$$
\R g_* v^*\mkern-2mu f^!\<F{\iso}
u^*\Rf f^!\<F
\longrightarrow u^*\<\<F\\[1.5pt]
$$
(see \eqref{bch} below), is identical with
the natural composite isomorphism
\begin{equation}\label{bciso}
v^*\<\<f^!\<F= v^!\<f^!\<F
\iso (fv)^!\<F=(ug)^!\<F\iso
g^!\<u^!\<F  =g^!\<u^*F.
\end{equation}

(N.B. The composite isomorphism \eqref{bciso} exists for \emph{any} $\sE$-maps
$f\<,g, u, v$ with $u$ and $v$ \'etale and $fv=ug$.)\va1

The case of finite\kf-type maps  is treated in \cite[Theorem 4.8.1]{li}, from the category-theoretic  viewpoint of Verdier and Deligne.%
\footnote{The frequent references in this paper to \cite{li} are due much more to 
the approach and convenience of that source than to its originality.}
 An extension to essentially-finite\kf-type maps is given
in \cite[\S5.2]{Nk}. The pseudo\-functor~$(-)^!$ expands so as to take values in $\Dqc(-)$ if one restricts
to proper maps or to \mbox{$\sE$\kf-maps} of finite flat dimension \cite[\S\S5.7--5.9]{AJL1}---and even without such restrictions if one changes `pseudo\-functor'\va1 to `oplax functor,'\va{1} i.e., one allows that for an $\sE$\kf-diagram  
$W\overset{\lift .7,g\>,}\lto X\overset{\lift 1.15,f\>,}\lto Z $\va{.5} the associated map $(fg)^!\to g^!f^!$ need not be an isomorphism, see \cite{Nm14}. For flat $\sE$-maps---the maps with which we shall be mainly concerned---the agreement of the oplax $(-)^!$  with the preceding pseudofunctor results from 
\cite[Proposition 13.11]{Nm14}.
\end{cosa}

\begin{cosa}
(Verdier's isomorphism, fundamental class and the ``Ideal Theorem.") In~the original working out  (\cite{RD}, amended in  
\cite{C1},  \cite{C2}) 
of the duality theory conceived by Grothendieck and Verdier, the main result, Corollary~3.4 on~p.\,383, is roughly that in the presence of residual complexes,%
\footnote{The theory of residual complexes presented in \cite[Chapters 6 and 7\kf]{RD} is considerably generalized in \cite{LNS}.} 
when confined to \mbox{finite\kf-type} scheme-maps and complexes with coherent homology the pseudo\-functor $(-)^!$ is obtained by pasting together concrete realizations of its restriction\- to smooth maps and to finite maps. This is a special case of the ``Ideal Theorem" in \cite[p.\,6]{RD}.\looseness=-1

For \emph{finite} $f\colon X\to Z$, the canonical concrete realization is induced by the usual sheafified duality isomorphism (see \eqref{sheaf dual}):\va2
\begin{equation}\label{d_f}
\R\fst f^!\<F = \R\fst \R\sHom(\OX, f^!\<F\>)\iso\R\sHom(\fst\OX, F\>)\quad(F\in\Dqcpl(Z)).
\end{equation}

For \emph{formally smooth} $f\colon X\to Z$ in $\sE$,  of relative dimension $d$ (section \ref{various}),
a canonical concrete realization is given by an isomorphism
\begin{equation}\label{e_f}
f^!\<F\iso \Omega^{\>d}_f[d\>]\otimes_\sX \<f^*\<\<F\qquad(F\in\Dqcpl(Z)),
\end{equation}
where $\Omega^{\>d}_f$ is the sheaf of relative $d$\kf-forms and ``$[d\>]$" denotes $d$\kf-fold translation (shift) in $\D(X)$.

\pagebreak[3]
The initial avatar of such an isomorphism uses a trace map for residual comp\-lexes (that are assumed to exist). 
In particular, when $f$ is proper there results an 
explicit (but somewhat complicated) description of Serre\kf-Grothendieck duality.  
(See e.g., \cite[\S3.4]{C1}.)

A definition of an isomorphism  with the same source and target as~\eqref{e_f}, but not requiring residual complexes, was given by Verdier in \cite[proof of Theorem 3]{V}. We review and expand upon this classical isomorphism in~\S\ref{V}. In particular, Proposition~\ref{v and otimes} explicates compatibility of the isomorphism with derived tensor product.\va1

For any \emph{flat} $\sE$\kf-map $f\colon X\to Z$, let $\delta=\delta_f\colon X\to X\times_ZX$ be the diagonal of $f$. There is a $\Dqcpl$-map \va{-6}
\[
\fc{\<f}\colon\delta^*\<\<\delta_*f^*\to f^!\<,
\]
the \emph{fundamental class of} $f\<$, from the relative Hochschild functor to the relative dualizing functor, see Definition~\ref{def-fc}.

The fundamental class is compatible with derived tensor product, see Proposition~\ref{c and tensor}.

When the flat map~$f$ is \emph{equidimensional of relative dimension}~$d$, applying the homology functor $H^{-d}$ to $\fc{\<f}$
leads to a map $c^{}_{\<f}$ from the target of the map~\eqref{e_f}  to the source
(Section~\ref{equidim} below.) 
Propositon 2.4.2 in \cite{AJL2} asserts that if, moreover, $f$~is \emph{formally smooth,} then $c^{}_{\<f}$ is an isomorphism. The proof uses Verdier's isomorphism, and of course begs the question of
whether that isomorphism is inverse to $c^{}_{\<f}$.\va2

\emph{That this does hold is our main result,}\va2 Theorem~\ref{v and c}.
\looseness=-1

\vskip2pt
Thus,  the map $c^{}_{\<f}$ extends the inverse of Verdier's isomorphism, from the class of formally smooth $\sE$\kf-maps to the class of arbitrary flat equidimensional $\sE$\kf-maps.

\begin{subrems}
(a) The above discussion has been limited, for simplicity, to~$\Dqcpl$; but the results will be established for $\Dqc\>$.\va1

(b) The isomorphisms in play are not quite canonical: there are choices involved that affect 
the resulting homology maps up to sign.  For example, for a scheme~$Y$ and an $\OY$-ideal~$\CI$, there are two natural identifications of $\CI/\CI^2$ with $\stor^{\OY}_1\<\<(\OY/\CI,\> \OY\</\CI\>)$, one the negative of the other; and we will have to choose one of them (see \eqref{I/I^2}\:\emph{ff}.) We will also have to assign different roles to the two projections of $X\times_Z X$ to $X$, necessitating another arbitrary choice. See also \cite[\S7.1]{Sa}. 

Our choices minimize sign considerations, under the constraint of respect for the usual triangulated structure
on the derived\- sheaf\kf-hom functor (see the remarks following equations \eqref{Hom-shift} and~\eqref{H and Hom}).
Other choices might have done as well.\va1

(c) In the present vein, we propose it as a nontrivial exercise (that, as far as we know, no one has carried out) to specify the relation between the above ``initial avatar" of the fundamental class and
Verdier's isomorphism. (Both isomorphisms have been known for fifty years or so.)\looseness=-1
\end{subrems}

We believe that the fundamental class is important enough, historically and technically, to merit the kind of scrutiny it gets in this paper. 
For~instance, familiarity with various of its aspects 
could well prove useful in establishing  that the pseudo\-functor~$(-)^!$  together
with the  isomorphisms in \eqref{d_f} and \eqref{e_f} satisfy VAR 1--6  (mutatis mutandis) in \cite[p.\,186]{RD}---a version over all of~$\sE$ of the ``Ideal Theorem." 
It is hoped that a full treatment of this application will materialize
in the not-too\kf-distant future. 

\end{cosa}

\begin{small}
\begin{cosa}\label{history}
(Additional background: fundamental class and residues.)
More history and motivation behind the fundamental class can be found in \cite[\S0.6]{AJL2}. \looseness=-1

A preliminary version of the fundamental class, with roots in the dualizing properties of 
differentials on normal varieties, appears in \cite[p.\,114]{G60}, followed by some brief hints about  connections with residues. 
In~\cite{blue}, there is a concrete treatment of the case
when  $S = \spec(k)$ with $k$ a perfect field and $f\colon X\to \spec k$  an integral algebraic 
$k$-scheme. The principal result (``Residue Theorem") reifies $c^{}_{\<f}$ as a globalization of the local
residue maps at the points of~$X\<$,  leading\- to explicit versions of local and global duality and the relation between them. This is generalized to certain maps of noetherian schemes
in \cite{HS}. 

The close relation between the fundamental class and residues
becomes\- clearer, and more general, over formal schemes, where local and global 
duality merge into a single theory with fundamental classes and residues
conjoined.  (See \mbox{\cite[\S5.5]{Mexico};} a complete exposition has yet to appear.)

\end{cosa} 
\end{small}

\section{Preliminaries}\label{prelims} In this section we describe those parts of the duality machinery,
along with a few of their basic properties, that we will subsequently use. The reader is advised 
initially to skip to the next section, referring back to this one as needed.

\begin{cosa}\label{abbrev}

Unless otherwise specified, we will  be working exclusively with functors between full subcategories of categories 
of the form $\D(X)$ (see beginning of~\S\ref{motive}); so to reduce clutter we will write: \va2

$\bullet\ \>\fst\,$ instead of $\R\fst$\quad ($f$ a scheme\kf-map), \va1

$\bullet\ f^*$ instead of $\LL f^*$\quad ($f$ a scheme\kf-map), \va1

$\bullet\ \otimes$ instead of $\Otimes{}$. \va1

$\bullet\ \sHom$ instead of $\R\sHom$. (This is the derived \emph{sheaf-hom} functor.)\va3

\end{cosa}

\begin{cosa}\label{adjunctions}

For any scheme\kf-map $f\colon X\to Z$,\va{.6} the functor $\fst\colon\D(X)\to\D(Z)$ is \emph{symmetric monoidal\kf;} this entails, in particular,  a functorial map 
\begin{equation}\label{symmon}
\fst E_1\otimes_Z\fst E_2\to\fst(E_1\otimes_\sX E_2)\qquad (E_i\in\D(X)), 
\end{equation}
see e.g., \cite[3.4.4(b)]{li}.

The functor $f^*\colon\D(Z)\to\D(X)$ is \emph{pseudofunctorially}\va{-2} left-adjoint to $\fst$ \cite[3.6.7(d)]{li}.  This means that for scheme-maps $X\xto{\:f\:}Z\xto{\:g\:}W$,\va{1.5} the canonical isomorphisms
$(gf)_*\iso g_*f_*$ and $f^*g^*\iso(gf)^*$ are \emph{conjugate}, see \cite[3.3.7(a)]{li}.

For fixed $B\in\D(X)$, the functor $-\otimes_\sX\< \<B$ is left-adjoint to $\sHom_X(B,-)$ 
(see e.g., \cite[(2.6.1)$'\>$]{li}). The resulting counit map
\begin{equation}\label{eval}
\ev(X,B,C)\colon\sHom_X(B,C)\otimes_\sX B\to C\qquad(C\in\D(X))
\end{equation}
is referred to as ``evaluation."
\end{cosa}

\begin{cosa}\label{various}
A scheme\kf-map $f\colon X\to Z$ is \emph{essentially of finite type} if every $z\in Z$ has an affine open neighborhood $\spec A$ whose inverse image is a union of finitely many affine open subschemes 
$\spec B_i$ such that each $B_i$ is a localization of a finitely generated $A$-algebra.  

Throughout, $\sE$ will be the category of essentially-finite\kf-type separated maps of noetherian schemes. 

The category $\sE$ is closed under scheme\kf-theoretic fiber product.
 
An $\sE$\kf-map $f\colon X\to Z$ is  \emph{essentially smooth} if it is flat and it has geometrically regular fibers; or equivalently,  if $f$ is \emph{formally smooth,} that is, for each $x\in X$ and $z\set f(x)$, 
the local ring $\CO_{\<\<X\mkern-2.5mu,\>x}$ is formally smooth over $\CO_{Z\<\<,\>z}$ for the discrete topologies, see \cite[17.1.2, 17.1.6]{EGA4} and \cite[Chapter 0, 19.3.3, 19.3.5(iv) and 22.6.4 \mbox{a)$\Leftrightarrow$ c)}]{EGA4}. (For this equivalence, which involves only local properties, it can be
assumed that $X=\spec A$ and $Z=\spec B_S$ where $B$ is a finite-type $A$-algebra and $S\subset B$ 
is a multiplicatively closed subset; then the relevant local rings are the same for $f$ or for
the finite\kf-type map  $g\colon\spec B\to\spec A$, so that one only needs the equivalence for $g$, as given by \cite[6.8.6]{EGA4}.)

For essentially smooth $f\colon X\to Z$, the diagonal map $X\to X\times_ZX$ is a regular immersion, see \cite[16.9.2]{EGA4},  \cite[16.10.2]{EGA4} and \cite[16.10.5]{EGA4}---whose proof is valid for $\sE$\kf-maps.\va1

Arguing as in
\cite[17.10.2]{EGA4}, one gets that for essentially smooth  $f$  the relative
differential sheaf $\>\Omega_f$  is locally free over $\OX$; moreover, when $f$ is of finite type the rank of $\>\Omega_f$ at $x\in X$ is the dimension at~$x$ of the fiber~$f^{-\<1}\<f(x)$. For any essentially smooth  $f$, the (locally constant) rank of~$\>\Omega_f$ will be  referred to as the \emph{relative dimension of}\kern1.5pt $f$. 

An $\sE$\kf-map is \emph{essentially \'etale} if it is essentially smooth and has relative dimension~0
(cf.~\cite[17.6.1]{EGA4}).
\end{cosa}
  
\begin{cosa}\label{perfect} 
An $\OX$-complex $E$ ($X$ a scheme) is \emph{perfect}
if each $x\in X$ has an open neighborhood $U=U_x$ such that the restriction $E|_U$ is 
$\D(U)$-isomorphic to  a bounded
complex of finite\kf-rank free $\OU$-modules.

When applied to $\sE$\kf-maps, the term ``perfect"  means
``having finite tor-dimension" (or ``finite flat dimension").

Perfection of maps is preserved under composition and under flat base change, see e.g., 
\cite[p.\,243, 3.4 and p.\,245, 3.5.1 ]{Il}.
\end{cosa}

\begin{cosa}\label{^times}
Let  $\Dqcpl(X)\subset\Dqc(X)\subset\D(X)$ be as 
at the beginning of~\S\ref{motive}.

For any $\sE$-map $f\colon X\to Z\<$, there exists a functor 
$f^\times\colon \D(Z)\to\Dqc(X)$ that is bounded below and right-adjoint to 
$\Rf\>$. (See e.g., \cite[\S4.1]{li}.) In particular, with $\id_\sX$ the identity map of $X$ one has a functor
$\id_{\sX}^\times$, the \emph{derived quasi-coherator,} right-adjoint to the inclusion
$\Dqc(X)\hookrightarrow\D(X)$. 

For any complexes $A$ and $B$ in $\Dqc(X)$, and with notation as in \S\ref{abbrev}, set
\begin{equation}\label{RHqc}
\Hqc{\sX}(A,B)\set \id_\sX^{\<\times}\<\sHom_\sX(A,B)\in\Dqc(X).
\end{equation}

$\Dqc(X)$ is a  (symmetric monoidal) closed category, with multiplication  (derived) $\otimes_{\sX}$ 
and internal hom  $\Hqc{\sX}$
(cf.~e.g., \cite[3.5.2(d)]{li}).

For $C\in\Dqc(X)$, the counit map is a $\D(X)$-isomorphism $\id_\sX^\times\<C\iso  C$.
So for $A$, $B\in\D(X)$,\va{.5} the counit map $\Hqc{\sX}(A,B)\to\sHom_\sX(A,B)$
is an \emph{isomorphism} when\va{-.3} $\sHom_\sX(A,B)\in\Dqc(X)$---for example, when $B\in\Dqcpl(X)$ and the cohomology sheaves $H^iA$ are coherent for all $i$,\va{.5} vanishing for $i\gg 0$,
see  \cite[p.\,92, 3.3]{RD}.

\end{cosa}

\begin{cosa}\label{maps}
We will use some standard functorial maps, gathered together here, that are associated to an $\sE$\kf-map $f\colon X\to Z$  and objects $E_\bullet\in\Dqc(X)$, \mbox{$F_\bullet\in\Dqc(Z)$,} 
$G_\bullet\in\Dqc(Z)$.
For the most part, the definitions of these maps emerge category-theoretically from Section~\ref{adjunctions},
cf.~\cite[\S3.5.4]{li}.\va2

{\bf(a)} The map
\begin{equation}\label{* and otimes}
f^*(F\otimes_Z G)\lto f^*\<\<F\otimes_\sX\< f^*G
\end{equation}
that is adjoint to the natural composite
\[
F\otimes_Z G\lto \fst f^*\<\<F\otimes_Z \fst f^*\<\<G\xto{\eqref{symmon}} \fst(f^*\<\<F\otimes_\sX \<f^*\<\<G).
\]
The map \eqref{* and otimes} is an \emph{isomorphism} \cite[3.2.4]{li}.

\vskip3pt
{\bf(b)} The \emph{sheafified adjunction isomorphism}:
\begin{equation}\label{sheafadj}
\fst\sHom_{X}(f^*G, E)\iso\sHom_Z(G, f_*E),
\end{equation}
right-conjugate \cite[\S1.6]{ILN}, for each  fixed $G$, 
to the isomorphism~\eqref{* and otimes}.

\vskip3pt
{\bf(c)} The \emph{projection isomorphisms,} see, e.g., \cite[3.9.4]{li}:
\begin{equation}\label{projn}
\fst E\otimes_{Z}\< G\iso\fst(E\otimes_{X}\<  f^*\<\<G)\quad\textup{and}\quad
G\Otimes{Z}\< \fst E\iso\fst(f^*\<\<G\Otimes{X} \< E).
\end{equation}
These are, respectively, the natural composites
\begin{align*}
&\fst E\otimes_Z\< G\lto\fst E\otimes_Z\< \fst f^*\<\<G\xto{\eqref{symmon}} \fst(E\>\otimes_X\<f^*\<\<G), \\
&G\otimes_Z\< \fst E\lto \fst f^*\<\<G\otimes_Z\< \fst E\xto{\eqref{symmon}}\> \fst(f^*\<\<G\>\otimes_X\<  E).
\end{align*}
(The definitions make sense for arbitrary scheme-maps, though in that generality the composites are not always isomorphisms.)
\vskip3pt
{\bf(d)} The map
\begin{equation}\label{f^*Hom}
f^*\sHom_{Z}(G, F\>) \lto \sHom_{\sX}(f^*\<\<G, f^*\<\<F\>)
\end{equation}
that is $f^*$-$\>\fst$ adjoint to the natural composite map
\[
\sHom_{Z}(G, F\>) \lto \sHom_{Z}(G, \fst f^*\<\<F\>)\underset{\eqref{sheafadj}\>}\iso
\fst\sHom_{\sX}(f^*\<\<G, f^*\<\<F\>).
\]
This  is an isomorphism if the map $f$ is perfect, $G$ is homologically bounded-above, with coherent homology sheaves, and $F\in\Dqcpl(Z)$, see  \cite[Proposition~4.6.6]{li}.

The map \eqref{f^*Hom} is  $\sHom\>$-$\>\otimes$ adjoint to the natural composite map
\[
f^*\sHom_{Z}(G,F\>)\otimes_\sX f^*G\underset{\eqref{* and otimes}}\iso f^*(\sHom_{Z}(G,F\>)\otimes_Z G)
\lto f^*F,
\]
see \cite[Exercise 3.5.6(a)]{li}.

\vskip3pt

\pagebreak[3]

{\bf(e)} The \emph{duality isomorphism}: 
\begin{equation}\label{sheaf dual}
\zeta(E,F\>)\colon\fst\Hqc{\sX}(E, f^\times\<\< F\>)\iso \Hqc{Z}(\fst E,F\>),
\end{equation}
right-conjugate, for each fixed~$E$  to the projection isomorphism 
\[
\fst(f^*G\Otimes{\sX} E)\osi G\Otimes{Z}\fst E\>.
\]

\vskip3pt
{\bf (f)} The bifunctorial map
\begin{equation}\label{chi}
\chi^{}_{\<\<f}(F,G)\colon f^!\<F\otimes_\sX f^* \<\<G \to  f^!\< (F\otimes_Z G), 
\end{equation}
defined in \cite[13.1, 13.2 and 13.3\kf]{Nm14} (with `$\sigma\>$' instead of `$\chi$'), and shown in \cite[13.11]{Nm14} to be an \emph{isomorphism} 
whenever the map $f$ is perfect. 

When $f$ is proper,  $f^!$ is right-adjoint to $\fst$ and the map $\chi^{}_{\<\<f}$ is adjoint to the natural composite map
\[
\fst(f^!\<F\otimes_\sX f^*G)\underset{\eqref{projn}}\iso \fst f^!\<F\otimes_Z G\lto F\otimes_Z G.
\]
\begin{small}
In particular, $\chi^{}_{\<\<f}(\OZ,G)$ identifies with a map of triangulated functors
\[
\ush{f} G\set f^!\OZ\otimes_\sX f^*\<\<G\to f^!G.
\]
Note however that whereas the isomorphism $f^!(G[1])\iso (f^!G)[1]$ associated with the triangulated
structure on $f^!$ is the identity map, 
the same is not true of $\ush{f}$. (See \cite[just before 1.5.5]{li}.)
\par
\end{small}
\vskip3pt
{\bf(g)} The  \emph{base change map}
\begin{equation}\label{bch}
\beta_\clubsuit(F\>)\colon v^*\<\<f^!\<F\lto g^!u^*\<\<F
\end{equation}
associated to a fiber square  in~$\sE$:
\begin{equation*}
\CD
X'@>v>> X\\
@VgV\mkern 20mu V @VVfV \\
Z' @>\lift4.8,\displaystyle\clubsuit,>u> Z
\endCD
\end{equation*} 
with $f$ any $\sE$\kf-map and $u$ (hence $v$) flat. This map is defined
in \cite[11.6 and 11.7]{Nm14}, and denoted $\theta:\mathfrak{R}\longrightarrow\mathfrak{S}$, where $\mathfrak{R}$, $\mathfrak{S}$ are (repectively) the maps
taking the cartesian square to $v^*f^!$ and $g^!u^*$.

In general, this is not an isomorphism. But
if $F\in \Dqcpl(Z)$ then this is the isomorphism of \cite[Thm.\,4.8.3]{li} and \cite[Theorem 5.3]{Nk};
and for any $F\in \Dqc(Z)$, if $f$ is perfect then this is the isomorphism of \cite[\S5.8.4]{AJL1}.
(It actually suffices for $g$ to be perfect, see \cite[11.13(i)]{Nm14}.)

\medbreak
 {\bf(h)} The bifunctorial map
\begin{equation}\label{tensor and Hom}
\kappa_\sX(E_1,E_2,E_3)\colon\sHom_\sX(E_1,E_2)\otimes_\sX E_3\to\sHom_\sX(E_1,E_2\otimes_\sX E_3)
\end{equation}
adjoint to the natural composite---with ev the evaluation map \eqref{eval}---
\[
\sHom_\sX(E_1,E_2)\otimes_\sX E_3\otimes_\sX E_1\iso
\sHom_\sX(E_1,E_2)\otimes_\sX E_1\otimes_\sX E_3\xto{\!\via\>\ev}
E_2\otimes_\sX E_3\>.
\]

The map \eqref{tensor and Hom} is an \emph{isomorphism} if the complex
$E_1$ is \emph{perfect}.  Indeed,
the question is local on $X\<$, so one can assume that $E_1$ is a bounded
complex of finite\kf-rank free $\OX$-modules, and conclude via a simple induction
(like  that in the second-last
paragraph in the proof of \cite[4.6.7]{li}) on the number of degrees
in which $E_1$ doesn't vanish.

Similarly, \eqref{tensor and Hom} is an isomorphism
if $E_3$ is perfect.

\end{cosa}

The  projection map~\eqref{projn} is compatible with derived tensor product, in the following sense:

\begin{lem}\label{projassoc}
Let\/  $f\colon X\to Z$ be a scheme map, $E\in\D(X),$ $F,\,G\in\D(Z)$. The following natural diagram, where $\otimes$ stands for\/ $\otimes_\sX$ or\/ $\otimes_Z\>,$ commutes.
\[
\def\1{$(\fst E\otimes F\>)\otimes G$}
\def\2{$\fst (E\otimes f^*\<\<F\>)\otimes G$}
\def\3{$\fst \big((E\otimes f^*\<\<F\>)\otimes f^*G\big)$}
\def\4{$\fst E\otimes (F\otimes G)$}
\def\5{$\fst\< \big(E\otimes f^*(F\otimes G)\big)$}
\def\6{$\fst\< \big(E\otimes (f^*F\<\<\otimes f^*G)\big)$}
 \bpic[xscale=4.5,yscale=1.3]
  \node(11) at (1,-1){\1};
  \node(12) at (1.95,-1){\2};
  \node(13) at (3,-1){\3};
  
  \node(21) at (1,-2){\4};
  \node(22) at (1.95,-2){\5};
  \node(23) at (3,-2){\6};

 %rows
   \draw[->] (11) -- (12) node[above=1pt, midway, scale=.75]{\eqref{projn}} ;
   \draw[->] (12) -- (13) node[above=1pt, midway, scale=.75]{\eqref{projn}} ;
   
   \draw[->] (21) -- (22) node[below=1pt, midway, scale=.75]{\eqref{projn}} ; 
   \draw[->] (22) -- (23) node[below=1pt, midway, scale=.75]{\eqref{* and otimes}} ; 
 
 %columns
   \draw[->] (11) -- (21) node[left=1pt, midway, scale=.75]{$\simeq$} ;   
   \draw[->] (13) -- (23) node[right=1pt, midway, scale=.75]{$\simeq$} ; 

 \epic
\]  
\end{lem}

\begin{proof}
After substituting for each instance of \eqref{projn} its definition, and recalling \eqref{symmon},
one comes down to
proving commutativity of the border of the following diagram---whose maps are the obvious ones:
\[\mkern-5mu
\def\1{$(\fst E\<\otimes\< F\>)\<\otimes\< G$}
\def\2{$\fst (E\<\otimes\< f^*\<\<F\>)\<\otimes\< G$}
\def\3{$\fst \big((E\<\otimes\< f^*\<\<F\>)\<\otimes\< f^*G\big)$}
\def\4{$\fst E\<\otimes\< \<(\<F\<\<\otimes\< G)$}
\def\5{$\fst\< \big(E\<\otimes\< f^*(F\<\otimes\< G)\big)$}
\def\6{$\fst\< \big(E\<\otimes\< (f^*\<\<F\<\otimes\< f^*G)\big)$}
\def\7{$(\fst E\<\otimes\< \fst f^*\<\< F\>)\<\otimes\< G$}
\def\8{$(\fst E\<\otimes\< \fst f^*\<\< F\>)\<\otimes\< \fst f^*G$}
\def\9{$\fst E\<\otimes\< (\fst f^*\<\< F\<\otimes\< G)$}
\def\ten{$\fst (E\<\otimes\< f^*\<\<F\>)\<\otimes\< \fst f^*G$}
\def\lvn{$\fst E\<\otimes\< (\fst f^*\<\< F\<\otimes\< \fst f^*G)$}
\def\twv{$\fst E\<\otimes\< \fst (f^*\<\<F\<\otimes\< f^*G)$}
\def\thn{$\fst E\<\otimes\< \fst f^*\<(F\<\otimes\< G)$}
 \bpic[xscale=4.65, yscale=1.2]
 
  \node(11) at (1,-1){\1};
  \node(12) at (2.1,-1){\7};
  \node(13) at (3,-1){\2};
  
  \node(22) at (2.45,-2){\8};
   
  \node(31) at (1,-3){\4};
  \node(32) at (1.75,-3){\9};
  \node(33) at (3,-3){\ten};

  \node(42) at (2.1,-4){\lvn}; 
  \node(43) at (3,-4){\3};
 
  \node(51) at (1,-5){\thn}; 
  \node(52) at (2.1,-5){\twv};
  
  \node(61) at (1,-6){\5};
  \node(63) at (3,-6){\6};
  
  %rows
    \draw[->] (11) -- (12) node[left=1pt, midway, scale=.75]{$$} ;   
    \draw[->] (12) -- (13) node[left=1pt, midway, scale=.75]{$$} ;  
    
    \draw[->] (31) -- (32) node[left=1pt, midway, scale=.75]{$$} ;  
   
    \draw[->] (51) -- (52) node[left=1pt, midway, scale=.75]{$$} ;   
   
    \draw[->] (61) -- (63) node[left=1pt, midway, scale=.75]{$$} ;  
   
  %columns
    \draw[->] (11) -- (31) node[left=1pt, midway, scale=.75]{$$} ;   
    \draw[->] (31) -- (51) node[left=1pt, midway, scale=.75]{$$} ;  
    \draw[->] (51) -- (61) node[left=1pt, midway, scale=.75]{$$} ;  
   
    \draw[->] (12) -- (22) node[left=1pt, midway, scale=.75]{$$} ;   
    \draw[->] (12) -- (32) node[left=1pt, midway, scale=.75]{$$} ;
    \draw[->] (31) -- (51) node[left=1pt, midway, scale=.75]{$$} ;  
    
    \draw[->] (32) -- (42) node[right=1pt, midway, scale=.75]{$$} ;
    \draw[->] (42) -- (52) node[right=1pt, midway, scale=.75]{$$} ;
   
    \draw[->] (13) -- (33) node[left=1pt, midway, scale=.75]{$$} ;   
    \draw[->] (33) -- (43) node[left=1pt, midway, scale=.75]{$$} ;  
    \draw[->] (43) -- (63) node[left=1pt, midway, scale=.75]{$$} ;

 %oblique  
   \draw[->] (12) -- (22) node[left=1pt, midway, scale=.75]{$$} ; 
   \draw[->] (22) -- (33) node[left=1pt, midway, scale=.75]{$$} ;
   \draw[->] (22) -- (42) node[left=1pt, midway, scale=.75]{$$} ;
   \draw[->] (52) -- (63) node[left=1pt, midway, scale=.75]{$$} ;

  %labels
   \node at (1.5,-4) {\circled6};
   \node at (2.7,-5) {\circled7};
 \epic
\]

 The commutativity of subdiagram \circled6 follows directly from the definition 
of~\eqref{* and otimes}. Commutativity of \circled7 is
given by symmetric monoidality of $\fst\>$, see~\S\ref{adjunctions}. 
Commutativity of the other subdiagrams
is pretty well obvious, whence the conclusion.
\end{proof}

For the map \eqref{tensor and Hom} we'll need a ``transitivity" property---an instance\- of the 
Kelly-Mac\,Lane coherence theorem \cite[p.\,107, Theorem 2.4]{KlM}:

\begin{lem}\label{trans}
With $\kappa\set\kappa_\sX$ as in \eqref{tensor and Hom}$,$
$\kappa(E_1,E_2,E_3\otimes E_4)$ factors as
\begin{align*}
\sHom_\sX(E_1,E_2)\otimes E_3\otimes E_4
&\xto{\kappa(E_1,E_2, E_3)\otimes\id}
\sHom_\sX(E_1,E_2\otimes E_3)\otimes E_4\\
&\xto{\kappa(E_1,E_2\otimes E_3, E_4)}
\sHom_\sX(E_1,E_2\otimes E_3\otimes E_4).
\end{align*}
\end{lem}

\begin{proof} The assertion results from commutativity of the following natural diagram---where 
$[-,-]\set\sHom_\sX(-,-),$  $\otimes\set\otimes_\sX,$ and various associativity isomorphisms are omitted:
\[\mkern-6.5mu
\def\1{$[E_1,E_2]\<\otimes\<\< E_3\<\otimes\<\< E_4\<\otimes\<\< E_1$}
\def\2{$[E_1,E_2\<\otimes\<\< E_3]\<\otimes\<\< E_4\<\otimes\<\< E_1$}
\def\3{$[E_1,E_2]\<\otimes\<\< E_3\<\otimes\<\< E_1\<\<\otimes\<\< E_4$}
\def\4{$[E_1,E_2\<\otimes\<\< E_3]\<\otimes\<\< E_1\<\<\otimes\<\< E_4$}
\def\5{$[E_1,E_2\<\otimes\<\< E_3\<\otimes\<\< E_4]\<\otimes\<\< E_1$}
\def\6{$[E_1,E_2]\<\otimes\<\< E_1\<\otimes\<\< E_3\<\<\otimes\<\< E_4$}
\def\7{$E_2\<\otimes\<\< E_3\<\otimes\<\< E_4$}
 \bpic[xscale=4.5,yscale=1.2]
 
  \node(11) at (1,-1)[scale=.95]{\1} ;
  \node(13) at (1,-3)[scale=.95]{\2};

  \node(21) at (2,-1)[scale=.95]{\3};  
  \node(22) at (2,-2)[scale=.95]{\4};
  \node(23) at (2.07,-3)[scale=.95]{\5};

  \node(31) at (3,-1)[scale=.95]{\6};
  \node(33) at (3,-3)[scale=.95]{\7};

 %rows
   \draw[->] (11) -- (21) node[above, midway, scale=0.75]{$\Iso$} ; 
   \draw[->] (21) -- (31) node[above, midway, scale=0.75]{$\Iso$} ;  
   
    \draw[->] (13) -- (23) node[below=1pt, midway, scale=0.75]{$\via\kappa$}; 
    \draw[->] (23) -- (33) node[below=1pt, midway, scale=0.75]{$\ev$};

  %columns
   
   \draw[->] (11) -- (13)  node[left=1pt, midway, scale=0.75]{$\via\kappa$};
   
   \draw[<-] (22) -- (21) node[right=1pt, midway, scale=0.75]{$\via\kappa$};

   \draw[->] (31) -- (33)   node[right=1pt, midway, scale=0.75]{$\via\>\ev$};

 %oblique
  \draw[->] (13) -- (22) node[above=-2.5pt, midway, scale=0.75]{$\simeq\mkern40mu$} ;            
  \draw[->] (22) -- (33) node[above=-2.5pt, midway, scale=0.75]{$\mkern80mu\via\>\ev$}; 
 
 %labels
   \node at (1.35,-1.8) {\circled1};
   \node at (2.65,-1.8) {\circled2};
   \node at (2,-2.55) {\circled3};
   
  \epic
\]
Commutativity of subdiagram \circled1 is clear; and that of \circled2 and \circled3 follow from the 
definitions of $\kappa(E_1,E_2,E_3)$ and $\kappa(E_1,E_2\otimes E_3, E_4)$, respectively.
\end{proof}

\begin{cosa} Let $\delta\colon X\to Y$ be an $\sE$\kf-map, and $p\colon Y\to X$ 
a scheme\kf-map such that 
$p\>\delta=\id_\sX$. We will be using the bifunctorial isomorphism 
\[
\psi=\psi(\delta,p, E, F\>)\colon \delta^*\<\<\delta_*E\otimes_\sX F\iso \delta^*\<\<\delta_*(E\otimes_\sX F\>)
\qquad(E,F\in\Dqc(X))
\]
that is defined to be the natural composite 
\begin{equation}\label{psi}
\def\1{$\delta^*\<\<\delta _*\>E\otimes_\sX F$}
\def\3{$\delta^*\<\<\delta_*(E\otimes_\sX F\>)$}
\def\4{$\delta^*\<\<\delta _*\>E\otimes_\sX \delta^*\<p^*\<F$}
\def\5{$\delta^*(\delta _*\>E\otimes_Y p^*\<F\>)$}
\def\6{$\delta^*\<\<\delta_*(E\otimes_\sX \delta^*\<p^*\<F\>)$}
\CD
 \bpic[xscale=3.85,yscale=1.3]
 
  \node(11) at (1,-1){\1} ;
  \node(13) at (3,-1){\3};
  
  \node(21) at (1,-2){\4};  
  \node(22) at (1.98,-2){\5};
  \node(23) at (3,-2){\6};

 %rows
    \draw[->] (21) -- (22) node[above, midway, scale=0.75]{$\Iso$} 
                                    node[below=1pt, midway, scale=0.75]{\eqref{* and otimes}}; 
    \draw[->] (22) -- (23) node[above, midway, scale=0.75]{$\Iso$} 
                                    node[below=1pt, midway, scale=0.75]{\eqref{projn}};  
  
  %columns
   
   \draw[->] (11) -- (21)  node[left=1pt, midway, scale=0.75]{$\simeq$};
   
   \draw[<-] (13) -- (23)  node[right=1pt, midway, scale=0.75]{$\simeq$};

  \epic
  \endCD
\end{equation}
(Cf.~\cite[(2.2.6)]{AJL2}.)

\smallskip
The rest of this subsection brings out properties of $\psi$ needed later on.

\begin{sublem} \label{psi and iso} With preceding notation,  the isomorphism
$\psi(\delta,p, E, \delta^*\<p^*\<\<F\>)$ factors as
\[
\delta^*\<\delta_*E\otimes_\sX \delta^*\<p^*\<\<F
\underset{\eqref{* and otimes}}\iso
\delta^*(\delta_*E\otimes_Y p^*\<\<F\>)
\underset{\eqref{projn}}\iso
\delta^*\delta_*(E\otimes_\sX \delta^*\<p^*\<\<F\>).
\]
\end{sublem}

\begin{proof} This results from the fact that the natural composites\va2
\[
\delta^*\<\<\delta_*E\<\otimes F
\!\iso\!
\delta^*\<\<\delta_*E\<\otimes\delta^*\<p^*\<\<F
\xto{\!\psi(\delta,\>\>p,\> E\<,\> \delta^*\<p^*\!F\>) }
\delta^*\<\<\delta_*(E\<\otimes\delta^*\<p^*\<\<F\>)
\!\iso\!
\delta^*\<\<\delta_*(E\<\otimes F\>)
\]
and 
\[
\delta^*\<\<\delta_*E\<\otimes F
\!\iso\!
\delta^*\<\<\delta_*E\<\otimes\delta^*\<p^*\<\<F
\xto{\!\eqref{projn}
{\lift-.3,\smallcirc,}
\eqref{* and otimes}}
\delta^*\<\<\delta_*(E\<\otimes\delta^*\<p^*\<\<F\>)
\!\iso\!
\delta^*\<\<\delta_*(E\<\otimes F\>)
\]
are both equal to $\psi(\delta,p, E, F\>)$ (the first by functoriality of $\psi$, and the second by definition).
\end{proof}

\goodbreak

\begin{subcor} 
   \label{psi and epsilon}
The following natural diagram commutes.
\[
\def\1{$\delta^*\<\<\delta _*E\otimes F$}
\def\2{$E\otimes F$}
\def\3{$\delta^*\<\<\delta_*(E\otimes F\>)$}
 \bpic[xscale=2.5,yscale=1.3]
 
  \node(11) at (1,-1){\1} ;
  \node(13) at (3,-1){\3};
  
  \node(22) at (2,-2){\2};

 %rows
   \draw[->] (11) -- (13) node[above, midway, scale=0.75]{$\psi$}; 
    
 %oblique
 
   \draw[->] (11) -- (22)  ;
   \draw[->] (13) -- (22)  ;

  \epic
\]

\end{subcor}

\begin{proof} One can replace $F$ by the isomorphic complex~
$\delta^*\<p^*\<\<F\<$, whereupon the assertion follows from Lemma~\ref{psi and iso} via 
\cite[3.4.6.2]{li}.
\end{proof}

\begin{sublem}\label{transpsi}
For $E,F_1,F_2\in\Dqc(X),$ with\/ $\psi(-,-)\set \psi(\delta,p, -, -)$\va1
 and $\otimes\set\otimes_{\sX},$
the isomorphism $\psi(E,\>F_1\otimes\< F_2\>)$ factors $($modulo associativity isomorphisms$\>)$ as
\begin{align*}
\delta^*\<\<\delta_*E\otimes F_1\otimes F_2
&\xto[\psi(E, \>\>F^{}_1)\>\>\otimes\,\id\>\>]{\Iso} \delta^*\<\<\delta_*(E\otimes F_1)\otimes F_2\\
&\xto[\psi(E\>\otimes F^{}_1,\>\>F^{}_2\>)]{\Iso}\delta^*\<\<\delta_*(E\otimes F_1\otimes F_2\>)
\end{align*}
\end{sublem}

\begin{proof} The Lemma asserts commutativity the border of the next diagram, in which $\otimes$
stands for $\otimes_\sX$ or $\otimes_Y$, and the maps are the obvious ones:\va3
\[\mkern-5mu
\def\1{$\delta^*\<\<\delta_*E\<\otimes\< F_1\<\otimes\< F_2$}
\def\2{$\delta^*(\delta_*E\<\otimes\< p^*\<\<F_1)\<\otimes\< F_2$}
\def\3{$\delta^*\<\big(\delta_*E\<\otimes\< p^*(F_1\<\otimes\< F_2\>)\big)$}
\def\4{$\delta^*\<\<\delta_*(E\<\otimes\< \delta^*\<p^*\<\<F_1)\<\otimes\< F_2$}
\def\5{$\delta^*\<\<\delta_*\<\big(E\<\otimes\< \delta^*\<p^*(F_1\<\otimes\< F_2\>)\big)$}
\def\6{$\delta^*\<\<\delta_*(E\<\otimes\< F_1)\<\otimes\< F_2$}
\def\7{$\delta^*\<\<\delta_*(E\<\otimes\< F_1\<\otimes\< F_2\>)$}
\def\8{$\delta^*\<\<\delta_*(E\<\otimes\< F_1\<\otimes\< \delta^*\<p^*\<\<F_2\>)$}
\def\9{$\,\delta^*(\delta_*E\<\otimes\< p^*\<\<F_1)\<\otimes\< \delta^*\<p^*\<\<F_2$}
\def\ten{$\delta^*(\delta_*E\<\otimes\< p^*\<\<F_1\<\otimes\< p^*\<\<F_2\>)$}
\def\lvn{$\delta^*\<\<\delta_*(E\<\otimes\< \delta^*\<p^*\<\<F_1)\<\otimes\< \delta^*\<p^*\<\<F_2\:$}
\def\twv{$\delta^*(\delta_*(E\<\otimes\< \delta^*\<p^*\<\<F_1)\<\otimes\< p^*\<\<F_2\>)$}
\def\thn{$\delta^*\<\<\delta_*\<\big(E\<\otimes\< \delta^*(p^*\<\<F_1\<\otimes\< p^*\<\<F_2\>)\big)$}
\def\frn{$\,\delta^*\<\<\delta_*(E\<\otimes\< \delta^*\<p^*\<\<F_1\<\otimes\< \delta^*\<p^*\<\<F_2\>)$}
\def\ffn{$\delta^*\<\<\delta_*(E\<\otimes\< F_1)\<\otimes\< \delta^*\<p^*\<\<F_2$}
\def\sxn{$\delta^*(\delta_*(E\<\otimes\< F_1)\<\otimes\< p^*\<\<F_2\>)$}
 \bpic[scale=.8, xscale=3.8, yscale=1.8]
 
  \node(11) at (1,-.85)[scale=.9]{\1};
  \node(14) at (4,-.85)[scale=.9]{\2};
  
  \node(22) at (3.15,-1.57)[scale=.9]{\9};

  \node(31) at (1,-3.25)[scale=.9]{\3};
  \node(32) at (1.73,-2.5)[scale=.9]{\ten};
  \node(33) at (3.15,-3.25)[scale=.9]{\lvn};
  \node(34) at (4,-2.5)[scale=.9]{\4};

  \node(42) at (1.73,-4.75)[scale=.9]{\thn};
  \node(43) at (4,-4.95)[scale=.9]{\ffn};
  \node(44) at (4,-3.77)[scale=.9]{\6};
 
  \node(51) at (1,-5.5)[scale=.9]{\5};       
  \node(52) at (2.127,-6.25)[scale=.9]{\frn};
  \node(53) at (3.15,-5.5)[scale=.9]{\twv};
   
  \node(61) at (1,-7.3)[scale=.9]{\7};
  \node(63) at (2.4,-7.3)[scale=.9]{\8};
  \node(64) at (4,-7.3)[scale=.9]{\sxn};

 %rows
   \draw[->] (11) -- (14) ;
   
   \draw[->] (31) -- (32) ; 

   \draw[->] (31) -- (32) ;
  
   \draw[->] (51) -- (42) ;
   \draw[<-] (43) -- (44) ;

   \draw[<-] (52) -- (53) ;
    
   \draw[<-] (61) -- (63) ;   
   \draw[<-] (63) -- (64) ;  
   
 %columns
   \draw[->] (11) -- (31) ;   
   \draw[->] (31) -- (51) ; 
   \draw[->] (51) -- (61);

   \draw[->] (32) -- (42) ;
   \draw[->] (42) -- (52) ;

   \draw[->] (33) -- (43) ; 
        
   \draw[->] (14) -- (34) ;   
   \draw[->] (34) -- (44) ;

 %oblique 
   \draw[<-] (22) -- (14) ; 
   \draw[->] (22) -- (32) ;  
   \draw[->] (22) -- (33) ; 
   \draw[->] (34) -- (33) ; 
   \draw[->] (32) -- (53) ; 
   \draw[->] (33) -- (53) ; 
   \draw[->] (52) -- (63) ;
   \draw[->] (53) -- (64) ;  
   \draw[->] (43) -- (64) ; 
     
  %labels
   \node at (1.73,-1.6) {\circled1};
   \node at (1.73,-6.8) {\circled3};
   \node at (1.98,-4.15) {\circled2};   
   
  \epic
\]

Commutativity of the unlabeled subdiagrams is easy to verify.

\pagebreak[3]
Subdiagram \circled1 expands naturally as\va3
\[\mkern-3mu
\def\1{$\delta^*\<\<\delta_*E\<\otimes\< F_1\<\otimes\< F_2$}
\def\2{$\delta^*(\delta_*E\<\otimes\< p^*\<\<F_1)\<\otimes\< F_2$}
\def\3{$\delta^*\<\big(\delta_*E\<\otimes\< p^*(F_1\<\otimes\< F_2\>)\big)$}
\def\4{$\delta^*\<\<\delta_*E\<\otimes\< \delta^*\<p^*\<\<F_1\<\otimes\< F_2$}
\def\5{$\delta^*\<\<\delta_*E\<\otimes\< \delta^*\<p^*(F_1\<\otimes\< F_2\>)$}
\def\6{$\delta^*\<\<\delta_*E\<\otimes\< \delta^*(p^*\<\<F_1\<\otimes\< p^*\<\<F_2\>)$}
\def\7{$\delta^*\<\<\delta_*E\<\otimes\< \delta^*\<p^*\<\<F_1\<\otimes\< \delta^*\<p^*\<\<F_2$}
\def\9{$\,\delta^*(\delta_*E\<\otimes\< p^*\<\<F_1)\<\otimes\< \delta^*\<p^*\<\<F_2$}
\def\ten{$\delta^*(\delta_*E\<\otimes\< p^*\<\<F_1\<\otimes\< p^*\<\<F_2\>)$}
\bpic[xscale=4.45, yscale=1.8]
 
  \node(11) at (1,-1)[scale=.92]{\1};
  \node(12) at (2,-1)[scale=.92]{\4};
  \node(13) at (3,-1)[scale=.92]{\2};
  
  \node(21) at (1,-1.75)[scale=.92]{\5};
  \node(22) at (1.5,-2.5)[scale=.92]{\6};

  \node(225) at (2,-1.75)[scale=.92]{\7};
  
  \node(31) at (1,-3.25)[scale=.92]{\3};
  \node(32) at (2,-3.25)[scale=.92]{\ten};
  \node(33) at (3,-3.25)[scale=.92]{\9};

 %rows
   \draw[->] (11) -- (12) ;
   \draw[->] (12) -- (13) ;
   
   \draw[->] (21) -- (22) ; 
   \draw[->] (31) -- (32) ; 
   \draw[->] (33) -- (32) ; 
   
 %columns
   \draw[->] (11) -- (21) ;   
   \draw[->] (21) -- (31) ; 
  
%   \draw[->] (12) -- (22) ;
   \draw[->] (22) -- (32) ;
   
   \draw[->] (13) -- (33) ; 

 %oblique 
   \draw[<-] (22) -- (225) ; 
   \draw[->] (225) -- (33) ; 
   \draw[<-] (225) -- (12) ;

  %labels
   \node at (1.5,-1.375) {\circled4};
   \node at (2.15,-2.55) {\circled5};
   
  \epic
\]

For commutativity of \circled4, see the proof of \cite[3.6.10]{li}. For that of \circled5, see \cite[Example 3.4.4(b)]{li}. Commutativity of the other two subdiagrams is clear, and so \circled1 commutes.

Commutativity of \circled2 is given by Lemma~\ref{projassoc}.

Finally, for commutativity of \circled3, see again the proof of \cite[3.6.10]{li}.
\end{proof}

\begin{subcor} \label{psiO=id}
Upon identifying---as one may---$G\otimes\OX$ with\/ $G$ for all\/ $G\in\D(X),$ one has that\/ $\psi\set\psi(E,\OX)$ is the identity map of\/ $\delta^*\delta_*E$.\looseness=-1
\end{subcor}

\begin{proof}
The case $F_1=F_2=\OX$ of Lemma~\ref{transpsi} implies that $\psi^2=\psi$, whence 
$\psi=\psi^2\psi^{-\<1}=\psi\psi^{-\<1}=\id$.
\end{proof}

\end{cosa}

\section{Verdier's isomorphism}\label{V}
Recall from Section~\ref{abbrev} that, unless otherwise specified, all the functors that appear are functors between derived categories.

\begin{cosa}\label{V's thm}
Let $U\xto{i\mkern1.5mu}X\xto{f} Z$ be $\sE$\kf-maps, with $i$ essentially \'etale, and 
let \va3
\[
 \CD
\hbox to 0pt{\hss$Y\set$} @.\:U\times_{\<Z} X@>p^{}_2>> X\\
@.@Vp^{}_1VV @VVfV \\
@.U @>\lift4.8,\displaystyle\spadesuit\mkern20mu,>fi> Z
 \endCD
\]
be the resulting fiber square, where  $p^{}_1$ and $p^{}_2$ are the canonical projections. 
With $\delta_U$ the diagonal map and $\id_U$ the identity map of~$U\<$, the composite\begin{equation}\label{graph}
g\colon U \xto{\,\delta_U\,}U\times_{\<Z} U \xto{\<\id_U\!\times\> i\>\>}U\times_{\<Z}X
\end{equation}
is the graph of $i$, i.e., $p^{}_1\>g=\id_U$  and $p^{}_2\>g=i$.\va1
 
In Theorem 3 of \cite{V},   $fi$ is essentially smooth, of relative dimension, say,~$d$ (see Section~\ref{various} above);
and that theorem says  \emph{there is a\/ $\D(U)$-isomorphism} 
\[
(fi)^!\OZ\iso \Omega^{\>d}_{\<fi}[d\>]\set\big(\lift1.7,{\bigwedge},\lift2,{\<\<d},\, \Omega_{\<fi}\big)[d\>].
\] 

A slight elaboration of Verdier's approach produces, as follows, a func\-torial isomorphism
\begin{equation}\label{vf}
v^{}_{\<\<f\<,\>i}(F\>)\colon (fi)^!\<F \iso \Omega^{\>d}_{\<fi}[d\>]\otimes_U\< (fi)^{\<*}\<F
\qquad \big(F\in\Dqcpl(Z)\big).
\end{equation}

In particular, making the allowable identifications 
\[
\Omega^{\>d}_{\<fi}[d\>]\otimes_{\sX}\<(fi)^*\OZ=\Omega^{\>d}_{\<fi}[d\>]\otimes_{\sX}\<\OX
=\Omega^{\>d}_{\<fi}[d\>], 
\]
we consider $v^{}_{\<\<f\<,\>i}(\OZ)$ to be a map from $(fi)^!\OZ$ to 
$\Omega^{\>d}_{\<fi}[d\>]$.\va2

(In \S\ref{unbounded v}, the definition of $v^{}_{\<\<f\<,\>i}(F\>)$ will be extended to all $F\in\Dqc(Z)$.)\va2

The construction of $v^{}_{\<\<f\<,\>i}(F\>)$ uses two maps.\va{.6} The first is the following natural composite map $\vartheta$\va{.6}---whose definition needs only that $fi\>$ be \emph{ flat}---with 
$\chi^{}_{g}$  as in \eqref{chi}, and $\beta_\spadesuit$ as in \eqref{bch} (an isomorphism since $F\in\Dqcpl(Y)$):\looseness=-1
\stepcounter{equation}
\[
\begin{aligned}\label{first}
\vartheta\colon g^!\OY\otimes_U (fi)^!\<F
&\iso g^!\OY\otimes_U i^*\mkern-2.5muf^!\<F\\
&\iso g^!\OY\otimes_U g^*\<p_2^*\>f^!\<F\\
&\,\overset{\chi^{}_{g}}\lto\, g^!\<p_2^*\>f^!\<F \\
&\,\overset{\beta_\spadesuit\>\>}\lto\, g^!p_1^!\>(fi)^*\<\<F\\[3pt]
& \iso (fi)^*\<\<F\<.
 \end{aligned}\tag{\theequation}
\]
  If $fi$ is essentially smooth, then in \eqref{graph}, 
both $\id_U\!\times\, i$ and the regular immersion
$\delta_U$ are perfect, whence so is $g$ (see \S\ref{perfect}),  so that $\chi^{}_{g}$ is an isomorphism.%
\footnote{\kern1pt As $\delta_U$ is finite and $\id_U\!\times\, i$ is essentially \'etale, one can see this more concretely by showing that $\delta_{U\<\<*}(\chi)$ is isomorphic to the natural map\va{-2} 
\[
\sHom(\delta_{U\<\<*}\OU\<,\CO_{U\<\times_{\<\<Z}U})\otimes G\lto\sHom(\delta_{U\<\<*}\OU\<,G)\qquad(G\set (\id_U\!\times\, i)^{\mkern-1.5mu*}\<p_2^*f^!\<\<F\>).
\]
}
Thus, in this case, $\vartheta$ is an isomorphism.\va 1

The second is an isomorphism that holds when $fi$ is essentially smooth of relative dimension $d$,
\begin{equation}\label{fund2}
g^!\OY\iso \sHom^{}_U(\Omega^{\>d}_{\<fi\>},\>\OU)[-d\>],
\end{equation}  
described in \cite[p.\,180, Corollary 7.3]{RD} or \cite[\S2.5]{C1} via  
the ``fundamental\- local isomorphism"  and Cartan-Eilenberg resolutions.
An avatar~\eqref{def-fund2} of~\eqref{fund2} is reviewed in \S\ref{fundloc} below. 

The $\OU$-module $\Omega^{\>d}_{\<fi\>}$ is invertible, so the complexes
$\sHom^{}_U(\Omega^{\>d}_{\<fi\>},\>\OU)[-d\>]$ 
and~$\sHom^{}_U(\Omega^{\>d}_{\<fi\>}[d\>],\>\OU)$ are naturally isomorphic in $\D(U)$ to the
complex~$G$ which is $(\Omega^{\>d}_{\<fi\>})^{-\<1}$ in degree $d$ and 0 elsewhere. Modulo these isomorphisms, the  isomorphism
\begin{equation}\label{Hom-shift}
\sHom^{}_U(\Omega^{\>d}_{\<fi\>},\>\OU)[-d\>]\iso\sHom^{}_U(\Omega^{\>d}_{\<fi\>}[d\>],\>\OU),
\end{equation}
resulting from the usual triangulated structure on the functor $\sHom(-,\OU)$ is given by scalar multiplication in $G$ by $(-1)^{d^2+d(d-1)/2}=(-1)^{d(d+1))/2}$ (cf.~\cite[p.\,11, (1.3.8)]{C1}.) 

Let
\begin{equation}\label{gamma}
\gamma\colon g^!\OY\iso\sHom^{}_U(\Omega^{\>d}_{\<fi\>}[d\>],\>\OU)
\end{equation}
be the isomorphism obtained by composition from \eqref{fund2} and \eqref{Hom-shift}.

Thus when $fi$ is essentially smooth, so that $\Omega^{\>d}_{\<fi\>}$ is an \emph{invertible} $\OU$-module, one has a chain of natural functorial isomorphisms, the first being inverse to 
the evaluation  map ~\ref{eval}:
\begin{equation}\label{defv}
 \begin{aligned}
  (fi)^!\<F
  &\iso
  \sHom^{}_U(\Omega^{\>d}_{\<fi\>}[d\>],(fi)^!\<F\>) \otimes_U\Omega^{\>d}_{\<fi\>}[d\>]\\[2pt]
  &\underset{\eqref{tensor and Hom}}\iso
  \sHom^{}_U(\Omega^{\>d}_{\<fi\>}[d\>],\>\OU)\otimes_U (fi)^!\<F\otimes_U
   \Omega^{\>d}_{\<fi\>}[d\>]\\
  &\underset{\eqref{gamma}}\iso
  g^!\OY\otimes_U (fi)^!\<F \otimes_U\Omega^{\>d}_{\<fi\>}[d\>]\\[-3pt]
  &\underset{\eqref{first}}\iso
  (fi)^*\<\<F\otimes_U\Omega^{\>d}_{\<fi\>}[d\>]\iso\Omega^{\>d}_{\<fi\>}[d\>]\otimes_U (fi)^*\<\<F.
 \end{aligned}
 \end{equation}
The map~$v^{}_{\<\<f\<,\>i}(F\>)$ is defined to be the composition of this chain.
\pagebreak[3]
\end{cosa}

\begin{cosa}\label{fundloc} Expanding a bit on \cite[\S2, II]{LS}, we review the relation~ \eqref{fund3}
between the normal bundle of a regular immersion \mbox{$\delta\colon U\to W$} and the relative dualizing complex $\delta^!\OW$; and from that deduce the isomorphism~\eqref{fund2}.\va1

Let $\delta\colon U\to W$ be any  closed immersion of schemes, and $I$  the kernel of the associated surjective map $s\colon\OW\twoheadrightarrow\delta_{\<*}\OU$.
Then  
\[
\OU\cong H^0\delta^*\<\OW\xto{H^0\mkern-.5mu \delta^*\!s\>\>} H^0\delta^*\<\delta_{\<*}\OU\cong \OU,
\]
is an \emph{isomorphism.} 

\pagebreak[3]
The natural triangle
\[
\delta^*\mkern-1.5mu I\lto\delta^*\OW\xto{\delta^*\<\<s\,}\delta^*\<\delta_{\<*}\OU\overset{+\,}\lto
\]
gives rise to an exact sequence of $\OU$-modules
\[
\CD
H^{-1}\delta^*\OW@>>> H^{-1}\delta^*\<\delta_{\<*}\OU@>t>> H^0\delta^*\<I@>>> \ker(H^0\delta^*\<\<s)\\
@| @. @| @|\\
0 @. @.  I\</\<I^{\mkern.5mu\sst2} @. 0
\endCD
\]
Clearly, $t$  is an isomorphism, and so one has the natural  $\OU$-isomorphism 
\begin{equation}\label{I/I^2}
t^{\<-1}\colon I\</\<I^{\mkern.5mu\sst2}\iso H^{-1}\delta^*\<\delta_{\<*}\OU.
\end{equation}
(The isomorphism $t$ is induced by the projection $C[-1]\twoheadrightarrow P$, where
$K\to I$ is a flat resolution of $I$ and $C$ is the mapping cone of the composite map 
$K\to I\hookrightarrow \OW$. It is  the \emph{negative} of the connecting homomorphism 
\[
\stor_1^{\OW}\!(\delta_{\<*}\OU,\delta_{\<*}\OU)\to\stor_0^{\OW}\!(\delta_{\<*}\OU,I)
\]
usually attached to the natural exact sequence  $0\to I\to\OW\to\delta_{\<*}\OU\to 0$, see \cite[end of \S1.4]{li}.)

There is an alternating graded $\OU$-algebra structure on 
$\oplus_{n\ge 0}\>\>H^{-n}\delta^*\<\delta_{\<*}\OU$, induced by the natural product map
\[
\delta^*\<\delta_{\<*}\OU\otimes_U\>\delta^*\<\delta_{\<*}\OU
\<\<\iso\<\<\delta^*(\delta_{\<*}\OU\otimes_W\>\delta_{\<*}\OU)
\lto\delta^*\<\delta_{\<*}(\OU\otimes_U\OU)\<\<\iso\<\<\delta^*\<\delta_{\<*}\OU. 
\]
(Cf.~e.g., \cite[p.\,201, Exercise 9(c)]{Bo}.) Hence \eqref{I/I^2} extends uniquely to 
a homomorphism of graded $\OU$-algebras
\begin{equation}\label{fund1}
\oplus_{n\ge 0}\lift1.7,{\bigwedge},\lift2,{\<\<n},(I\</\<I^{\mkern.5mu\sst2})\lto 
\oplus_{n\ge 0}\>\>H^{-n}\>\delta^*\<\delta_{\<*}\OU.
\end{equation}

For example, over an affine open subset of $W\<$,\va{.6} if $I$ is generated by~a regular sequence of global 
sections\va{.5} $\mathbf t\set(t_1,t_2, \dots, t_d)$ then a finite free resolution of $\delta_*\OU$ is provided by the Koszul\va{.75} complex $K(\mathbf t)\set\otimes_{i=1}^d K_i$ where $K_i$ is the complex\va1 which is $\OW\xto{\lift.95,\<t_i\>\>,}\OW$ in degrees -1 and 0 and vanishes elsewhere; and there results
a $\D(U)$-map, 
\begin{equation}\label{fund1'}
\oplus_{n =0}^d\,\lift1.7,{\bigwedge},\lift2,{\<\<n},(I\</\<I^{\mkern.5mu\sst2})[n]\cong
\delta^*\<K(\mathbf t)\iso \delta^*\<\delta_{\<*}\OU.
\end{equation}
(Note that $\delta^*\<K(\mathbf t)$ is just the exterior algebra---with vanishing differentials---on 
$\OU^{\>n}$, which is isomorphic to $I/I^2$ via the natural map 
\[
\OW^{\>n}=K^{-1}(\mathbf t)\hookrightarrow I\subset\OW.)
\]
One verifies that applying the functor $H^{-n}$ to \eqref{fund1'} produces the degree~$n$ component of \eqref{fund1} (a map that  
\emph{does not depend on the choice of the generating family}~$\mathbf t$).
\va1

Next, for any $\OU$-complex $E$ with $H^e\<E=0$ for all $e<0$, the natural map $H^0E\to E$ induces a map
\begin{equation}\label{H0Hom}
H^0\>\sHom^{}_U\<(E, \OU)\lto H^0\sHom(H^0\<E,\OU)=:\sHom^0(H^0\<E,\OU).
\end{equation}
Hence for any integer~$n$ and any $\OU$-complex $F$ such that $H^e\<F=0$ for all~$e<-n$, one has the natural composite\va3
\begin{equation}
 \begin{aligned}\label{H and Hom}
H^{n^{\mathstrut}}\>\sHom^{}_U\<(F, \OU)&\iso H^0\>\sHom^{}_U\<(F, \OU)[n\>]\\
&\iso H^0\>\sHom^{}_U\<(F[-n\>], \OU)\mkern-12mu\\
&\underset{\eqref{H0Hom}}\lto \sHom^0(H^0(F[-n\>]),\OU)
\iso \sHom^0(H^{-n}F,\OU).
 \end{aligned}
\end{equation}
(The second 
map---whose source and~target are equal---is  multiplication by $(-1)^{n(n+1)/2}\>$: replace $\OU$ by a quasi-isomorphic injective complex, and take $p=0$, \mbox{$m=-n$} in the expression\va1 
$(-1)^{pm +m(m-1)/2}$ after \cite[p.\,11, (1.3.8)]{C1}.)\looseness=-1
\va1

Now if the closed immersion $\delta$ is \emph{regular, of codimension} $d$, i.e., 
$I$ is generated locally by regular sequences $\mathbf t$ of length~$d\>$, so that 
$\delta_{\<*}\OU$ is locally resolved by free complexes of the form $K(\mathbf t)$, then there results the
sequence of \emph{isomorphisms}
\begin{equation}
 \begin{aligned}\label{homologize}
 H^d\delta^*\<\delta_{\<*}\delta^!\OW
&\underset{\eqref{d_f}}\iso H^d\delta^*\sHom^{}_W\<(\delta_{\<*}\OU,\OW) \\[-2pt]
&\mkern-2.5mu\underset{\eqref{f^*Hom}}\iso \<H^d\>\sHom^{}_U\<(\delta^*\<\delta_{\<*}\OU, \OU)\\
&\,\underset{\eqref{H and Hom}}\iso\, \sHom^0_U\<(H^{-d}\>\delta^*\<\delta_{\<*}\OU,\OU)\\[-2pt]
&\,\underset{\eqref{fund1}}\iso\, \sHom^0_U\<(\lift1.7,{\bigwedge},\lift2,{\<\<d},(I\</\<I^{\mkern.5mu\sst2}),\OU).\\[2pt]
 \end{aligned}
\end{equation}
In particular, there is a canonical \emph{$\OU$-isomorphism}
\begin{equation*}
H^d\delta^*\<\delta_{\<*}\delta^!\OW
\iso \sHom^0_U\<\<\big(\lift1.7,{\bigwedge},\lift2,{\<\<d},\>(I\</\<I^{\mkern.5mu\sst2}),\OU\big) 
=:\nu^{}_{\delta}.
\end{equation*}
Moreover, since
$\delta_{\<*}\delta^!\OW\cong\sHom^{}_W(\delta_{\<*}\OU,\OW)$ has nonvanishing homology only in degree $d$, 
therefore the same holds for $\delta^!\OW$, whence the natural maps are isomorphisms
\[
\delta^!\OW[d\>]\iso H^d\delta^!\OW\osi H^d\delta^*\<\delta_{\<*}\delta^!\OW.
\]

Thus, \emph{when\/ $\delta$ is a regular immersion there is a canonical isomorphism}
\begin{equation}\label{fund3}
\boxed{\delta^!\OW\iso \nu^{}_\delta[-d\>].}
\end{equation}

\begin{small}
It is left to the interested reader to work out the precise relationship of~\eqref{fund3} to the similar isomorphisms in \cite[p.\,180, Corollary 7.3]{RD} and \cite[\S2.5]{C1}.\par
\end{small}

\smallskip
Now in \eqref{graph}, if $fi$ is essentially smooth of relative dimension $d$ then $\delta_U\colon U\to W\set U\times_{\<Z}U$
is a regular immersion and $\Omega^{\>d}_{\<fi\>}$ is locally free of rank one (see \S\ref{various}), so  the natural map  
$\sHom^0_U(\Omega^{\>d}_{\<fi\>},\>\OU)\iso\sHom^{}_U(\Omega^{\>d}_{\<fi\>},\>\OU)$
is an isomorphism. Thus,  
there are natural isomorphisms, with $Y\set U\times_ZX$,
\begin{equation}\label{def-fund2}
 \begin{aligned}
 g^!\CO_Y&\iso \delta_U^!(\id_U\!\times\>\>i)^!\CO_Y=\delta_U^!(\id_U\!\times\>\>i)^*\CO_Y
 =\delta_U^!\OW\\
 &\underset{\eqref{fund3}}\iso 
 \sHom^0_U\<(\lift1.7,{\bigwedge},\lift2,{\<\<d},(I\</\<I^{\mkern.5mu\sst2}),\OU)[-d\>]
 =\:\sHom^{}_U(\Omega^{\>d}_{\<fi\>},\>\OU)[-d\>].
 \end{aligned}
\end{equation}
The isomorphism \eqref{fund2} is defined to be the resulting composition.

\end{cosa}

\begin{cosa}

Though it is not \emph{a priori} clear,\va{.6}  $v^{}_{\<\<f\<,\>i}$ depends only on the map $fi$ 
and not on its factorization into $f$ and $i$. (In \cite{V}, $f$ is assumed proper.)\va{.6}
Indeed:

\begin{subprop} For\/ $U\xto{i\mkern1.5mu}X\xto{f} Z$ as in\/ \textup{\S\ref{V's thm},} if\/ $fi$ is essentially smooth then\/ 
$v^{}_{\<\<f\<\<,\>i}=v^{}_{\<\<fi,\>\>\id_U}.$
\end{subprop}

\begin{proof} There is 
a commutative $\sE$\kf-diagram, with $i$, $f$, $p^{}_1$, $p^{}_2$ and~$g$ as in~\S\ref{V's thm},
$\delta$~the diagonal map, $q^{}_1$, $q^{}_2$  the canonical projections, and $\spadesuit$, 
$\clubsuit$ and $\heartsuit$ labeling the front, top and rear faces of the cube,
respectively. (These three faces are fiber squares; and since $fi$ is flat therefore so are $p^{}_2$ and~$q^{}_2\>$.)

\[\mkern-45mu  
    \begin{tikzpicture}[xscale=1.065,yscale=.9
    ]
      \draw[white] (0cm,5.2cm) -- +(0: \linewidth)
      node (12) [black, pos = 0.2] {$U$}
      node (13) [black, pos = 0.45] {$W\set U\times_{\<Z} U$}
      node (16) [black, pos = 0.73] {$U$};
      \draw[white] (0cm,3.7cm) -- +(0: \linewidth)
      node (21) [black, pos = 0.27] {$Y\set U\times_{\<Z} X$}
      node (24) [black, pos = 0.55] {$X$}
      node (25) [black, pos = 0.6] {$\lift -.4,\displaystyle\heartsuit,$};

      \draw[white] (0cm,2cm) -- +(0: \linewidth)
      node (33) [black, pos = 0.45] {$U$}
      node (36) [black, pos = 0.73] {$Z$};
      \draw[white] (0cm,0.5cm) -- +(0: \linewidth)
      node (41) [black, pos = 0.27] {$U$}
      node (44) [black, pos = 0.55] {$Z$};
      
      \node (C) at (intersection of 13--33 and 21--24) { };
      \node (D) at (intersection of 33--36 and 24--44) { };
 
      \draw [->] (12) -- (13) node[auto, midway, scale=0.75]{$\delta$}; 
      \draw [->] (12) -- (21) node[left, midway, scale=0.75]{$g\mkern7mu$};    
      \draw [->] (21) -- (24) node[above, midway, scale=0.75]{$\mkern70mu p^{}_2$};
      \draw [->] (13) -- (16) node[auto, midway, scale=0.75]{$\mkern60mu q^{}_2$};
      \draw [-]  (33) -- (D)  node[auto, midway, scale=0.75]{ };
      \draw [->] (D)  -- (36) node[auto, near start, swap, scale=0.75]{$fi$};
      \draw [->] (41) -- (44) node[auto, swap, midway, scale=0.75]{$fi$};
      \draw [<-] (21) -- (13) node[above=-3.5pt, midway, scale=0.75]{\rotatebox{33}{$\mkern-30muj\set\<\id\<\<\times i\mkern-10mu$}};
      \draw [<-] (24) -- (16) node[auto, midway, scale=0.75]{$i\mkern-7mu$};
      \draw [double distance=2pt] (41) -- (33) ;
      \draw [double distance=2pt] (44) -- (36);
      \draw [<-] (41) -- (21) node[auto, midway, scale=0.75]{$p^{}_1$};
      \draw [-]  (C)  -- (13) node[auto, midway, scale=0.75]{ };
      \draw [<-] (33) -- (C)  node[auto, midway, scale=0.75]{$q^{}_1$};
      \draw [->] (24) -- (44) node[auto, midway, scale=0.75]{\raisebox{35pt}{$f$}};
      \draw [->] (16) -- (36) node[auto, midway, scale=0.75]{$fi$};
      
       \node at (6.2,4.5) [scale=.9]{$\clubsuit$};
       \node at (5.1,2.4) [scale=.9]{$\spadesuit$};
    \end{tikzpicture}
\]

Note that since $i$ is essentially \'etale therefore $i^!=i^*$ and $j^!=j^*\<$.

There is, as in \eqref{def-fund2}, a natural composite isomorphism
\[
\xi\colon g^!\OY\iso \delta^!j^!\OY =  \delta^!j^*\OY\iso\delta^!\OW.
\]

A detailed examination of the definition of $v^{}_{\<\<f\<,\>i}$ in \eqref{defv}, taking into account the definition \eqref{def-fund2} of \eqref{fund2}, shows that 
it will suffice to prove commutativity of the border of the following natural functorial diagram.\va2

%\begin{figure}
\begin{equation*}\label{bigfig}
\mkern-5mu
\def\1{$g^!\OY\<\otimes (fi)^!$}
\def\2{$g^!\OY\<\otimes i^*\mkern-2.5mu f^!$}
\def\3{$g^!\OY\<\<\otimes\< g^*p^*_2f^!$}
\def\4{$g^!p^*_2f^!$}
\def\5{$g^!p^!_1(fi)^*$}
\def\6{$(fi)^*$}
\def\7{$\delta^!\OW\<\otimes (fi)^!$}
\def\8{$\delta^!\OW\<\otimes \delta^*\<q_2^*(fi)^!$}
\def\9{$\delta^!q_2^*(fi)^!$}
\def\ten{$\delta^!q_1^!(fi)^*$}
\def\lvn{$\delta^!j^!\OY\<\otimes i^*\mkern-2.5mu f^!$}
\def\twv{$\delta^!j^*\OY\<\otimes i^*\mkern-2.5mu f^!$}
\def\thn{$\delta^!j^!\OY\<\<\otimes\< \delta^*\<j^*p_2^*f^!$}
\def\frn{$\delta^!\OW\<\otimes \delta^*\<q_2^*i^*\mkern-2.5mu f^!$}
\def\ffn{$\delta^!(j^!\OY\otimes j^*p_2^*f^!)$}
\def\sxn{$\delta^!(\OW\otimes q_2^*i^*\mkern-2.5mu f^!)$}
\def\svn{$\delta^!j^!p^*_2f^!=\delta^!j^*p^*_2f^!$}
\def\egn{$\delta^!q_2^*i^*\mkern-2.5mu f^!$}
\def\ntn{$\delta^!j^!p^!_1(fi)^*$}
\def\twy{$\delta^!(p^{}_1j)^!(fi)^*$}
\def\twn{$\delta^!(\OW\<\otimes q_2^*(fi)^!\>)$}
\def\twt{$\delta^!(\OW\otimes j^*p_2^*f^!)$}
\def\tth{$g^!(\OY\<\<\otimes\< g^*p^*_2f^!)$}
 \bpic[scale=.89, xscale=3.75,yscale=1.75]
 
  \node(11) at (1,-1.23){\1};
  \node(14) at (4,-1.29){\7};
  
  \node(21) at (1,-2){\2};
  \node(22) at (2.12,-2){\lvn};
  \node(23) at (3.28,-2){\twv};

  \node(31) at (1,-3.17){\3};
  \node(32) at (2.12,-3.17){\thn};
  \node(33) at (3.28,-3.17){\frn};
  \node(34) at (4,-2.43){\8};

  \node(41) at (1,-4.35){\tth};
  \node(42) at (2.12,-4.35){\ffn};
  \node(425) at (2.7, -5.1){\twt};
  \node(43) at (3.28,-4.35){\sxn};
  \node(44) at (4,-3.62){\twn};
        
  \node(51) at (1,-5.8){\4};
  \node(52) at (2.12,-5.8){\svn};
  \node(53) at (3.28,-5.8){\egn};
  \node(54) at (4,-5.8){\9};
  
  \node(61) at (1,-6.8){\5};
  \node(62) at (2.12,-6.8){\ntn};
  \node(63) at (3.22,-6.8){\twy};
  \node(64) at (4,-6.8){\ten};

  \node(72) at (2.5,-7.7){\6};

 %rows
   \draw[->] (11) -- (14) node[above=1pt, midway, scale=0.75]{$\via\>\xi$};
   
   \draw[->] (21) -- (22) ; 
   \draw[double distance = 2pt] (22) -- (23) ;

   \draw[->] (31) -- (32) ;
   \draw[->] (32) -- (33) ;
  
   \draw[->] (42) -- (43) ;

   \draw[->] (51) -- (52) ;
   \draw[->] (52) -- (53) ;
   \draw[->] (53) -- (54) ;
   
   \draw[->] (61) -- (62) ;   
   \draw[->] (62) -- (63) ;
   \draw[double distance = 2pt] (63) -- (64) ;  
   
 %columns
   \draw[->] (11) -- (21) node[left=1pt, midway, scale=0.75]{$\simeq$};   
   \draw[->] (21) -- (31) node[left=1pt, midway, scale=0.75]{$\simeq$};
   \draw[->] (31) -- (41) node[left=1pt, midway, scale=0.75]{$\chi^{}_{g}$}; 
   \draw[double distance = 2pt] (41) -- (51);
   \draw[->] (51) -- (61) node[left, midway, scale=0.75]{$g^!\<\beta_{\spadesuit}$}; 

   \draw[->] (32) -- (42) node[left=.5pt, midway, scale=0.75]{$\chi^{}_{\delta}$};
   \draw[->] (42) -- (52) node[left=.5pt, midway, scale=0.75]{$\delta^!\chi^{}_{\<j}$};
   \draw[->] (52) -- (62) node[left, midway, scale=0.75]{$\via \beta_{\spadesuit}$}; 

   \draw[->] (23) -- (33) node[right=1pt, midway, scale=0.75]{$\simeq$}; 
   \draw[->] (33) -- (43) node[right=1pt, midway, scale=0.75]{$\chi^{}_{\delta}$}; 
   \draw[double distance=2pt] (43) -- (53) ; 
     
   \draw[->] (14) -- (34) node[right=1pt, midway, scale=0.75]{$\simeq$};   
   \draw[->] (34) -- (44) node[right=1pt, midway, scale=0.75]{$\chi^{}_{\delta}$};
   \draw[double distance=2pt] (44) -- (54) ; 
   \draw[->] (54) -- (64) node[right=1pt, midway, scale=0.75]{$\delta^!\<\beta_{\>\heartsuit}$}; 

% %oblique 
   \draw[<-] (14) -- (23);  
   \draw[->] (33) -- (34) ; 
   
    \draw[->] (3.48, -4.18) -- (3.76,-3.85);
   
   \draw[->] (61) -- (72) node[above=-8pt, midway, scale=0.75]{$\simeq\mkern40mu$}; 
   \draw[->] (64) -- (72) node[above=-8pt, midway, scale=0.75]{$\mkern40mu\simeq$};  
  
   \draw[->] (42) -- (425) ;
   \draw[<-] (43) -- (425) ;
   \draw[double distance = 2pt] (52) -- (425) ;
   
  %labels
   \node at (2.12,-2.6) {\circled1};
   \node at (1.545,-3.79) {\circled2};
   \node at (2.24,-4.83) {\circled3};   
   \node at (3.12,-6.35){\circled4};
   \node at (2.5,-7.25) {\circled5};

  \epic
\end{equation*}
%\end{figure}

Commutativity of subdiagram \circled1 (resp.~\circled5) results from the pseudofunctoriality
of~$(-)^*$ (resp.~$(-)^!\>\>$).\va1

Commutativity of \circled2 is given by transitivity of $\chi$ with respect to composition of maps, 
see \cite[13.4]{Nm14} (or~\cite[Exercises 4.7.3.4(d) and 4.9.3(d)]{li}).\va1

Commutativity of \circled3 follows from the fact that for any $F\in\Dqcpl(Y)$, the following natural diagram commutes:\va3
\[
 \bpic[scale=.89, xscale=3.75,yscale=1.9]
  \node(11) at (1,-1){$j^!\OY\otimes j^*\<\<F$};
  \node(12) at (2,-1){$j^*\<\<\OY\otimes j^*\<\<F$};
  \node(13) at (2.93,-1){$\OW\otimes  j^*\<\<F$};
  
  \node(21) at (1,-2){$j^!(\OY\otimes F\>)$};
  \node(22) at (2,-2){$j^*(\OY\otimes F\>)$};
  \node(23) at (2.93,-2){$ j^*\<\<F$};
  
  \draw[double distance = 2pt] (11) -- (12);
  \draw[double distance = 2pt] (21) -- (22);
  \draw[->] (12)--(13) node[above, midway, scale=0.75]{$\Iso$};
  \draw[double distance = 2pt] (22)--(23) ;
  
  \draw[->] (11)--(21) node[left=1pt, midway, scale=0.75]{$\chi^{}_{j}$};
  \draw[->] (12) --(22) node[right=1pt, midway, scale=0.75]{$\simeq$};
  \draw[double distance = 2pt] (13) -- (23);
  
  \node at (1.5,-1.54)[scale=.9]{\circled3$^{}_{\<1}$};
  \node at (2.5,-1.54)[scale=.9]{\circled3$^{}_2$};
 \epic
\]
For a sketch of the proof that \circled3$^{}_{\<1}$ commutes see \cite[4.9.2.3]{li}. As for commuta\-tivity 
of \circled3$^{}_2$, 
replacement of $F$ by a quasi-isomorphic flat complex reduces the problem to the context of ordinary (nonderived) functors, at which point the justification is left to the reader.

Commutativity of \circled4 is given by transitivity of $\beta$ with respect to juxtaposition of 
fiber squares (see \cite[Theorem 11.9]{Nm14} or \cite[Theorem 4.8.3]{li}), as applied to the following decomposition of the fiber square $\heartsuit$:
\[
\CD
\bullet @>q^{}_2>> \bullet\\
@VjV\mkern35mu\clubsuit V@VViV\\
\bullet @>p^{}_2>> \bullet\\
@Vp^{}_1V\mkern35mu\spadesuit V@VVfV\\
\bullet @>>fi> \bullet
\endCD
\]
Here one needs to use that
$\beta_\clubsuit$ is the canonical isomorphism $q_2^*i^*\iso j^*p^*_2$ (see \cite[11.4 and 11.5]{Nm14}
with $p=p'=\:$identity map in the diagram of \cite[11.4(i)]{Nm14}, or
\cite[4.8.8(i)]{li}.)

Commutativity of the remaining subdiagrams is easy to check, whence the assertion.
\end{proof}

\smallskip
Accordingly, we restrict henceforth
to the case\/ $U=X$ and $i=\id_\sX$.
The map $g\colon X\to X\otimes_Z X$ is then the diagonal.\va2
\end{cosa}

\begin{cosa}\label{unbounded v}
Again, let $f\colon X\to Z$  be an essentially smooth $\sE$\kf-map of relative dimen\-sion~$d$.
With $\chi^{}_{\<\<f}$ as in \eqref{chi},
\emph{define\/ $v^{}_{\!f}(F\>)$ for} $F\in\Dqc(Z)$ to be the composite isomorphism
\begin{equation}\label{def2v}
f^!\<\<F\xto[\lift1.3,\chi_{\!f}^{-\<1},]{\Iso} f^!\OZ\otimes_{\sX}f^*\<\<F 
\xto[\! v^{}_{\!f\<\<,\>\id_{\sX}}\!(\<\OZ\<\<)\>\otimes\>\>\id\>]{\Iso}
\Omega_{\<f}^d[d\>]\otimes_{\sX}\<f^*\<\<F.
\end{equation}

As before (just after \eqref{vf}), the identifications 
\[\Omega_{\<f}^d[d\>]\otimes_{\sX}\<f^*\OZ=\Omega_{\<f}^d[d\>]\otimes_{\sX}\<\OX
=\Omega_{\<f}^d[d\>], 
\]
allow us to consider $v^{}_{\!f}(\OZ)$ to be a map from $f^!\OZ$ to $\Omega_{\<f}^d[d\>]$.

\begin{subprop}\label{def v}
If\/ $F\in\Dqcpl(Z)$ then
$v^{}_{\<\<f}(F\>)= v^{}_{\<\<f\<,\>\>\id_{\sss \sX}}\<(F\>)$.
\end{subprop}

\begin{proof}
Let $g\colon X\to Y\set X\times_{\<Z}X$ be the diagonal map. 
Set $\omega\set\Omega_{\<f}^d[d\>]$ (a perfect complex), and set
$[A,B]\set \sHom_X(A,B)\ (A,B\in\D(X))$. Let $\kappa=\kappa_X(\omega,-,-)$  be as in~\eqref{tensor and Hom},  $\vartheta$ as in \eqref{first} and  $\gamma$ as in \eqref{gamma} (the~last two with $i\set\id_\sX$). Since $f$ is essentially smooth, all of these maps are isomorphisms.
By definition (see \eqref{defv}), Proposition~\ref{unbounded v} says that the border of the following diagram  (in which $\otimes\set\otimes_\sX$) commutes:

\[\mkern-5mu
\def\1{$f^!\<F$}
\def\2{$f^!\OZ\<\<\otimes\< f^*\<\<F$}
\def\3{$\omega\<\otimes\< f^*\<\<F$}
\def\4{$[\omega, f^!\<F]\<\otimes\< \omega$}
\def\5{$[\omega, f^!\OZ]\<\otimes\< \omega\<\otimes\< f^*\<\<F\mkern10mu$}
\def\6{$f^*\OZ\<\otimes\< \omega\<\otimes\< f^*\<\<F$}
\def\7{$[\omega, f^!\OZ\<\otimes\< f^*\<\<F\>]\<\otimes\< \omega$}
\def\8{$g^!\OY\<\otimes\< f^!\OZ\<\otimes\< \omega\<\otimes\< f^*\<\<F\mkern10mu$}
\def\9{$[\omega,\OX]\<\otimes\< f^!\OZ\<\otimes\< \omega\<\otimes\< f^*\<\<F$}
\def\ten{$[\omega,\OX]\<\otimes\< f^!\<F\<\otimes\< \omega$}
\def\lvn{$\mkern20mu[\omega,\<\OX]\<\<\otimes\! f^!\OZ\<\<\otimes\! f^*\<\<F\<\<\otimes\< \omega$}
\def\twv{$f^*\OZ\<\otimes\< f^*\<\<F\<\otimes\< \omega$}
\def\thn{$f^*\<\<F\<\otimes\< \omega$}
\def\frn{$g^!\OY\<\otimes\< f^!\OZ\<\otimes\< f^*\<\<F\<\otimes\< \omega\mkern10mu$}
\def\ffn{$g^!\OY\<\<\otimes\<\< f^!\<F\<\otimes\< \omega$}
\def\sxn{$\mkern16mu[\omega\<,\< f^!\OZ]\<\<\otimes\!f^*\<\<F\<\<\otimes\< \omega $}
 \bpic[scale=.75, xscale=4.4,yscale=1.65]
 
  \node(11) at (1.6,-.65){\2};
  \node(12) at (4,-.65){\2};
  \node(14) at (4,-3){\3};

  \node(20) at (1,-1.65){\1};  
  \node(21) at (1.6,-3){\7};
  \node(22) at (2.65,-2.2){\5};
  \node(23) at (3.35,-1.5){\8};

  \node(30) at (2.07,-3.95){\sxn};  
  \node(31) at (1,-3.95){\4};
  \node(33) at (3.35,-3.95){\6};

  \node(41) at (2.65,-4.8){\9};
  \node(42) at (3.35,-5.75){\twv};
  \node(43) at (4,-4.8){\thn};
    
  \node(51) at (1,-7.65){\ten};  
  \node(52) at (1.6,-6.7){\lvn};
  \node(53) at (3.35,-6.7){\frn};
  \node(54) at (4,-7.65){\ffn};
 
  %rows
   \draw[double distance=2pt] (11) -- (12) node[above=-.5pt, midway, scale=0.75]{$$};
   \draw[->] (51) -- (54) node[below=.5pt, midway, scale=0.75]{$\gamma^{-\<1}$};     
   \draw[->] (52) -- (53) node[above, midway, scale=0.75]{$\gamma^{-\<1}$};   
   \draw[->] (53) -- (54) node[above=-9pt, midway, scale=0.75]{$\chi^{}_{\<\<f}\mkern40mu$};

 %columns
   \draw[->] (11) -- (20) node[above=-11pt, midway, scale=0.75]{$\mkern25mu\chi^{}_{\<\<f}$};
   \draw[->] (20) -- (31) node[left=1pt, midway, scale=0.75]{$\simeq$};   
   \draw[->] (31) -- (51) node[left, midway, scale=0.75]{$\kappa^{-\<1}$};
   \draw[<-] (51) -- (52) node[above=-9pt, midway, scale=0.75]{$\mkern45mu\chi^{}_{\<\<f}$};    
   
   \draw[->] (11) -- (22) node[above=-.5pt, midway, scale=0.75]{$\mkern15mu\simeq$};   
   \draw[->] (22) -- (41) node[left, midway, scale=0.75]{$\kappa^{-\<1}$}; 
   \draw[->] (41) -- (52) node[right=1pt, midway, scale=0.75]{$\mkern5mu\simeq$}; 
    
   \draw[->] (41) -- (23) node[right=.5pt, midway, scale=0.75]{$\gamma^{-\<1}$}; 
   
   \draw[->] (23) -- (33) node[left, midway, scale=0.75]{$\vartheta$};
      
   \draw[->] (21) -- (52) node[left, midway, scale=0.75]{$\kappa^{-\<1}$};
   
   \draw[<-] (33) -- (42) node[right=.5pt, midway, scale=0.75]{$\simeq$};
   \draw[<-] (42) -- (53) node[left, midway, scale=0.75]{$\vartheta$} ;  

   \draw[->] (12) -- (14) node[left, midway, scale=0.75]
       {$v^{}_{\!f\<\<,\>\id_{\sX}}\!(\<\OZ\<\<)\<\<$}
       node[right, midway, scale=0.75]{$\<\otimes\<\id_{\mathstrut}$}; 
   \draw[<-] (14) -- (43) node[left=1pt, midway, scale=0.75]{$\simeq$};   
   \draw[<-] (43) -- (54) node[right=1pt, midway, scale=0.75]{$\vartheta$} ; 
 
 %oblique
  \draw[<-] (31) -- (21)  node[above=-11pt, midway, scale=0.75]{$\mkern25mu\chi^{}_{\<\<f}$};   
 
  \draw[<-] (21) -- (11) node[above, midway, scale=0.75]{$\mkern25mu\simeq$};

  \draw[<-] (30) -- (22) node[above=-2pt, midway, scale=0.75]{$\simeq\mkern25mu$}; 
  \draw[->] (30) -- (21) node[above=-4pt, midway, scale=0.75]{$\mkern30mu\kappa$}; 
  \draw[->] (2.02,-4.18) -- (1.665,-6.46)  node[right, midway, scale=0.75]{$\kappa^{-\<1}$};
    
  \draw[double distance=2pt] (42) -- (43) ; 
  
  \draw[double distance=2pt] (14) -- (33) ;     

 %labels
   \node at (2.65,-1.1) {\circled1};
   \node at (1.9,-2) {\circled2};
   \node at (1.8,-4.5) {\circled3};
   \node at (3.7,-6.25) {\circled4};
 
  \epic
\]

Commutativity of the unlabeled subdiagrams is straightforward to verify.

Commutativity of \circled1 follows from the definition \eqref{defv} of 
$v^{}_{\<\<f\<,\>\>\id_{\sss \sX}}\<(\OZ)$.

Commutativity of \circled2 is essentially the definition of the map $\kappa$.

Commutativity of \circled3 results from Lemma \ref{trans}.

For \circled4, it's enough to have commutativity of the functorial diagram
\[
\CD
g^!\OY\otimes_\sX\< f^!\<\OZ\otimes_\sX\< f^*\<\<@>{\<\<\via\>\chi^{}_{\<\<f}\>\>}>>g^!\OY\otimes_\sX f^!\<\<\\[-5pt]
@V\via\vartheta\< VV  @VV\vartheta V\\
f^*\OZ\otimes_\sX\< f^*\<\< @= f^*\<,\endCD
\]
that expands to the border of the next natural diagram, 
in which the omitted subscripts
on $\otimes$ symbols are the obvious ones, and, with reference to the standard fiber square
\begin{equation*}
 \CD
X\times_{\<Z} X@>p^{}_2>> X\\
@Vp^{}_1VV @VVfV \\
X @>\lift4.8,\displaystyle\diamondsuit,>f> \,Z\>,
 \endCD
\end{equation*} $\beta_{\diamondsuit}$ is as in \eqref{bch}:
\[\mkern-5mu
\def\1{$g^!\OY\<\otimes f^!\OZ\<\otimes f^*$}
\def\2{$g^!\OY\<\otimes f^!$}
\def\3{$g^!\OY\<\<\otimes\< g^*p^*_2f^!\OZ\<\<\otimes\< g^*p^*_2f^*$}
\def\4{$g^!\OY\<\<\otimes \<g^*p^*_2(f^!\OZ\<\otimes\<\<f^*\<)$}
\def\5{$g^!\OY\<\<\otimes\< g^*p^*_2f^!$}
\def\6{$g^!\OY\<\<\otimes\< g^*\<(p^*_2f^!\OZ\<\<\otimes\< p^*_2f^*\<)$}
\def\7{$g^!\OY\<\<\otimes g^*p^!_1f^*\<\OZ\<\<\otimes\< g^*p^*_2f^*$}
\def\9{$g^!\OY\<\otimes g^*p^!_1f^*$}
\def\ten{$g^!\OY\<\<\otimes\<g^*\<(p^!_1f^*\<\OZ\<\otimes p^*_1f^*\<)$}
\def\twv{$g^!p^!_1f^*\OZ\<\otimes g^*p^*_1f^*$}
\def\thn{$g^!(p^!_1f^*\OZ\<\otimes p^*_1f^*)$}
\def\frn{$g^!p^!_1f^*$}
\def\ffn{$f^*\OZ\<\otimes f^*$}
\def\sxn{$(p^{}_1g)^!\<\<f^*\<\OZ\<\otimes\< (p^{}_1g)^*\<\<f^*$}
\def\svn{$(p^{}_1g)^!\<(f^*\OZ\<\otimes\<\< f^*)$}
\def\egn{$(p^{}_1g)^!f^*$}
\def\ntn{$f^*$}
 \bpic[scale=.89, xscale=3.43,yscale=1.75]
 
  \node(11) at (1,-1){\1};
  \node(14) at (4,-1){\2};
  
  \node(23) at (2.7,-2.125){\4};
  \node(24) at (4,-2.125){\5};

  \node(31) at (1,-3.25){\3};
  \node(33) at (2.7,-3.25){\6};

  \node(41) at (1,-4.375){\7};
  \node(43) at (2.7,-4.375){\ten};
  
  \node(54) at (4,-3.75){\9};
    
  \node(61) at (1,-5.5){\twv};
  \node(63) at (2.7,-5.5){\thn};
  \node(64) at (4,-5.5){\frn};

  \node(71) at (1,-6.625){\sxn};
  \node(73) at (2.68,-6.625){\svn};
  \node(74) at (4,-6.625){\egn};

  \node(81) at (1,-7.75){\ffn};
  \node(84) at (4,-7.75){\ntn};
 
 %rows
   \draw[->] (11) -- (14) node[above, midway, scale=0.75]{$\via\chi^{}_{\<\<f}$};
   
   \draw[->] (23) -- (24) node[above=2pt, midway, scale=0.75]{$\via\chi^{}_{\<\<f}$};

   \draw[->] (31) -- (33) node[above=-.5pt, midway, scale=0.75]{$\Iso$};
  
   \draw[->] (41) -- (43) node[above=-.5pt, midway, scale=0.75]{$\Iso$};
   \draw[->] (43) -- (54) node[below=-2pt, midway, scale=0.75]{$\mkern60mu\via\chi^{}_{p^{}_1}$};

   \draw[->] (61) -- (63) node[below=1pt, midway, scale=0.75]{$\!\chi^{}_{g}$};   
   \draw[->] (63) -- (64) node[below=1pt, midway, scale=0.75]{$\via\chi^{}_{p^{}_1}$};
   
   \draw[->] (71) -- (73) node[above, midway, scale=0.75]{$\!\chi^{}_{p^{}_1g}$};
   \draw[double distance=2pt] (73) -- (74) ;

   \draw[double distance=2pt] (81) -- (84) ;

 %columns
   \draw[->] (11) -- (31) node[left=1pt, midway, scale=0.75]{$\simeq$};   
   \draw[->] (31) -- (41) node[left=1pt, midway, scale=0.75]{$\via\beta_{\>\diamondsuit}$}; 
    \draw[->] (41) -- (61) node[left=1pt, midway, scale=0.75]{$\via\chi^{}_{g}$};    
   \draw[->] (61) -- (71) node[left=1pt, midway, scale=0.75]{$\simeq$}; 
   \draw[->] (71) -- (81) node[left=1pt, midway, scale=0.75]{$\simeq$};;   

   \draw[->] (23) -- (33) node[right=1pt, midway, scale=0.75]{$\simeq$};   
   \draw[->] (33) -- (43) node[right=.5pt, midway, scale=0.75]{$\via\beta_{\>\diamondsuit}$};
   \draw[->] (43) -- (63) node[right=1pt, midway, scale=0.75]{$\via \chi^{}_{g}$}; 
   \draw[->] (14) -- (24) node[right=1pt, midway, scale=0.75]{$\simeq$};   
   \draw[->] (24) -- (54) node[right=1pt, midway, scale=0.75]{$\via\beta_{\>\diamondsuit}$}; 
   \draw[->] (54) -- (64) node[right=1pt, midway, scale=0.75]{$\chi^{}_{g}$}; 
   \draw[->] (64) -- (74) node[right=1pt, midway, scale=0.75]{$\simeq$}; 
   \draw[->] (74) -- (84) node[right=1pt, midway, scale=0.75]{$\simeq$}; 

 %oblique
  \draw[->] (11) -- (23) node[above=.5pt, midway, scale=0.75]{$\simeq$};     
  
  \draw[->] (81) -- (73) node[above=.4pt, midway, scale=0.75]{$\simeq$};   
  
 %labels
   \node at (1.55,-2.125) {\circled5};
   \node at (3.53,-2.95) {\circled6};
   \node at (1.84,-5) {\circled7};
   \node at (2.68,-6.12){\circled8};
   \node at (1.55,-7.08) {\circled9};

  \epic
\]

For commutativity of subdiagram \circled5, see the last two paragraphs of \S3.6 in \cite{li}.

Commutativity of \circled6 results from \cite[13.7]{Nm14} with
applied to the diagram $\diamondsuit$ (cf.~\cite[4.9.3(c)]{li}.)

Commutativity of \circled7, \circled8 and \circled9  are left mostly to the reader (cf.~\cite[Exercises 4.7.3.4(a), (d) and (b)]{li},
as modified in \cite[4.9.3(d)]{li}, describing the behavior 
of~$\chi$ vis-\`a-vis associativity of tensor product, composition of maps, and identity maps, respectively.) 
For more details on \circled8, see 
\cite[13.4]{Nm14}.\looseness=-1

Commutativity of the remaining subdiagrams is easy to check, whence the assertion.
\end{proof}
\end{cosa}
 
 \pagebreak[3]
\begin{subrem}[Base change]\label{base change c} Let there be given a fiber square in\/ $\sE$\va{1.5}
\begin{equation*}
\CD
X'@>g'>> X\\
@Vf'V\mkern 20mu V @VVfV\\ 
Z' @>>\lift1.2, g,> Z\\[3.5pt]
\endCD
\end{equation*}
in which\/ $f$ (and hence $f')$ is essentially smooth of relative dimension\/~$d.$ 

\noindent The~isomorphism $v^{}_{\<\<f}(\OZ)$ induces an $\OX$-isomorphism 
\[
\omega^{}_{\<\<f}\set H^{-d}f^!\OZ\iso H^{-d}\big(\Omega^d_f[d]\big)= \Omega^d_f\>,
\]
and similarly for $f'$. The resulting composite isomorphism\va2
\[
g'{}^* \omega^{}_{\<\<f}
\iso
g'{}^*\Omega^d_f 
\iso
\Omega^d_{f'}  
\iso
\omega^{}_{\<\<f'},
\]
\vskip2pt
\noindent is discussed in ~\cite[p.\,740, Theorem 2.3.5]{Sa}.
\end{subrem}

 The next Proposition expresses compatibility between $v^{}_{\!f}$
and (derived) tensor product.

\begin{subprop}\label{v and otimes} 
For any\/ $F_1,$ $F_2\in\Dqc(Z),$ the following diagram, 
where $\otimes\set\otimes_\sX,$ commutes.\va1
\[
\CD
f^!\<F_1\otimes\< f^*\<\<F_2 @>\!v^{}_{\<\<f}(F_1)\>\otimes\> \id\>\> >>   \Omega_{\<f}^d[d\>]\otimes\<f^*\<\<F_1\otimes\<f^*\<\< F_2\\
@V \chi^{}_{\<\<f}(F_1\<,\>\>F_2) VV  @VV  \textup{natural}\, V\\
 f^!(F_1\otimes F_2)  @>>v^{}_{\<\<f}(F_1\otimes F_2)\>   >
    \Omega_{\<f}^d[d\>] \otimes\<f^*\<(F_1\otimes F_2). \\[4pt]
 \endCD
\]

\end{subprop}

\begin{proof}[Proof \textup{(Sketch)}] The definition of $v^{}_{\<\<f}$ makes it enough to prove commuta\-tivity of the next diagram (expressing transitivity of $\chi^{}_{\<\<f}$ for  any $F_0\in\Dqc(Z)$), and then to take $F_0\set\OZ$.\va1
\[
\CD
f^!(F_0\otimes\<F_1) \otimes f^*\<\<F_2  @<\,\chi^{}_{\<\<f}(\<F_0\<,\>\>F_1\<)\>\otimes\>\id<<
      f^!\<F_0\otimes\<f^*\<\<F_1\otimes\<f^*\<\< F_2\\
@V \chi^{}_{\<\<f}(\<F_0\>\otimes\> F_1\<,\>\>F_2\<)  VV  @VV\textup{natural} V\\
f^!(F_0\otimes F_1\otimes F_2) @<<\,\chi^{}_{\<\<f}(\<F_0,\>\>F_1\>\otimes\> F_2\<)\<\< < 
     f^!\<F_0\otimes  f^*(F_1\otimes\<F_2). \\[4pt]
 \endCD
\]
For this, one reduces easily, via a compactification of $f\<$, to the case where $f$~is proper, a case dealt with (in outline) in
\cite[Exercise 4.7.3.4(a)]{li}.
\end{proof}

\enlargethispage{-5pt}
\pagebreak[3]
\section{The fundamental class} \label{fundclass}

\begin{cosa}\label{setup}

Let $f\colon X\to Z$ be a flat $\sE$\kf-map. 
Set $Y\set X\times_{\<Z}X$, let $\delta\colon X\to Y$ be the diagonal map, and $p^{}_i\colon Y\to X\ (i=1,2)$ the projections onto the first and second factors, respectively, so that we have the diagram, with fiber square~$\clubsuit\>\>$,\looseness=-1
\begin{equation}\label{club}
\CD
X @>\delta >> Y @>p^{}_2 >> X \\
@. @V p^{}_1 VV @VVfV \\
@. X @>\lift4.8,\displaystyle\clubsuit,>\lift1.1,f,> Z
\endCD
\end{equation}
The maps $p_i$ are flat.

There are maps of $\Dqc$-functors 
\begin{equation}\label{delta->p}
\mu_i\colon\delta_*\to p_i^!\qquad(i=1,2)
\end{equation}
adjoint, respectively, to the natural maps $\id=(p_i\delta)^!\to \delta^!p_i^!$. Thus $\mu_i$ is the natural composite map
\[
\delta_*=\delta^{}_*(p^{}_i\delta)^!\lto \delta^{}_*\delta^!p_i^!\lto p_i^!\>.
\]

Associated to $\clubsuit$ is the functorial \emph{base-change isomorphism} (see \eqref{bch})
\[
\beta=\beta_{\clubsuit}\colon p_2^*f^!\<\<F\iso p_1^!f^*\<F\qquad(F\in\Dqc(Z)).
\]

\begin{defn}\label{def-fc}
With preceding notation, 
the \emph{fundamental class of} $\<f,$
\[
\fc{\<f}\colon\delta^*\delta^{}_*\>f^*\to f^!,
\]  
a map between functors from $\Dqc(Z)$ to $\Dqc(X)$, is  given
by the composite 
\begin{equation*}\label{cf}
\delta^*\delta^{}_*\>f^*
\xto[\via\mu^{}_1]{}
 \delta^*{p^!_1}\>f^*
\xto[\>\>\delta^*\<\<\beta^{-\<1}\>]{\lift-.1,\Iso,}
\delta^*p_2^*\>f^! 
\xto[\!\textup{natural}\!]{\lift-.1,\Iso,} f^!.
\end{equation*}
\end{defn}

\noindent \emph{Remarks.}\ (Not used elsewhere). It results from \cite[2.5 and~3.1]{AJL2} that \emph{the~fundamental class commutes with essentially \'etale localization on~$X\<$.} That is, if $g\colon X'\to X$ is essentially \'etale
then $\fc{\<f\<g}$ is obtained from~$\fc{\<f}$ by applying the functor $g^*$ and then making canonical identifications. 

See also \cite[Theorem 5.1]{AJL2} for the behavior of $\>\fc{\<f}$ under flat base change.

These results imply that if $u\colon U\hookrightarrow X$ and $v\colon V\hookrightarrow Z$ are  open immersions such that $f(U)\subset V$, and $f^{}_{\<0}\colon U\to V$ is the restriction of $f$, then 
$u^*(\fc{\<f})$ can be identified with $\fc{\<{f^{}_{\<0}}}$.  Locally, then, $\fc{\<f}$ reduces to the fundamental class of a flat $\sE$-map of affine schemes, in which case a simple explicit description is given in \cite[Theorem 4.2.4]{ILN}.

\end{cosa}

 The next Proposition expresses compatibility between $\fc{\<f}$
and (derived) tensor product.

\begin{prop}\label{c and tensor}
For any\/ $F_1,$ $F_2\in\Dqc(Z),$ the following diagram, 
where $\otimes\set\otimes_\sX,$  $\chi^{}_{\<\<f}$ is as in \eqref{chi} 
and\/ $\psi\set\psi(\delta,p_1,f^*\<\<F_1,f^*\<\<F_2),$
commutes. 
\[
\def\1{$\delta^*\<\delta_*f^*\<\<F_1\otimes f^*\<\<F_2$}
\def\2{$f^!\<F_1\otimes\< f^*\<\<F_2$}
\def\3{$\delta^*\<\delta_*(f^*\<\<F_1\otimes f^*\<\<F_2) $}
\def\4{$\delta^*\<\delta_*f^*(F_1\otimes F_2) $}
\def\5{$f^!(F_1\otimes F_2)$}
\bpic[xscale=5,yscale=1.5]
 
  \node(11) at (1,-1){\1};
  \node(12) at (2,-1){\2};
  
  \node(21) at (1,-2){\3};
  
  \node(31) at (1,-3){\4};
  \node(32) at (2,-3){\5};
  
  %rows
   \draw[->] (11) -- (12) node[above, midway, scale=0.75]{$\fc{\<f}(F_1)\otimes\>\id$};
   
   \draw[->] (31) -- (32) node[below=1pt, midway, scale=0.75]{$\fc{\<f}(F_1\otimes F_2)$};

 %columns
   
   \draw[->] (11) -- (21) node[left=1pt, midway, scale=0.75]{$\psi$};
   \draw[->] (21) -- (31) node[left=1pt, midway, scale=0.75]{$\eqref{* and otimes}$}; 
 
   \draw[->] (12) -- (32) node[right=1pt, midway, scale=0.75]{$\chi^{}_{\<\<f}(F_1\<,\>\>F_2)$};
   \epic
\]

\end{prop}

\begin{proof}
The diagram expands, naturally, as follows, where $\alpha$ is induced by the composite $p_2^*f^*\iso(fp_2)^*=(fp_1)^*\iso p_1^*f^*$ and by the base-change 
isomorphism $\beta$.

\[\mkern-5mu
\def\1{$\delta^*\<\delta_*f^*\<\<F_1\otimes f^*\<\<F_2$}
\def\2{$\delta^*p_1^!f^*\<\<F_1\otimes f^*\<\<F_2$}
\def\3{$\delta^*p_2^*f^!\<F_1\otimes f^*\<\<F_2$}
\def\4{$f^!\<F_1\otimes\< f^*\<\<F_2$}
\def\5{$\delta^*(\delta_*f^*\<\<F_1\otimes p_1^*f^*\<\<F_2)$}
\def\6{$\delta^*(p_1^!f^*\<\<F_1\otimes p_1^*f^*\<\<F_2)$}
\def\7{$\delta^*(p_2^*f^!\<F_1\otimes p_2^*f^*\<\<F_2)$}
\def\8{$\delta^*\<\delta_*(f^*\<\<F_1\otimes \delta^*p_1^*f^*\<\<F_2)$}
\def\9{$\delta^*p_2^*(f^!\<F_1\otimes f^*\<\<F_2)$}
\def\ten{$\delta^*\<\delta_*f^*(F_1\otimes F_2)$}
\def\lvn{$\delta^*p_1^!f^*(F_1\otimes F_2)$}
\def\twv{$\delta^*p_2^*f^!(F_1\otimes F_2)$}
\def\thn{$f^!(F_1\otimes F_2)$}
\def\frn{$\delta^*p_1^!(f^*\<\<F_1\otimes f^*\<\<F_2)$}
\def\ffn{$\delta^*\<\delta_*(f^*\<\<F_1\otimes f^*\<\<F_2)$}
\bpic[scale=.89, xscale=3.65,yscale=1.45]
 
  \node(11) at (1,-1)[scale=.93]{\1};
  \node(12) at (2,-1)[scale=.93]{\2};
  \node(13) at (3,-1)[scale=.93]{\3};
  \node(14) at (4,-1)[scale=.93]{\4};
  
  \node(21) at (1,-3)[scale=.93]{\5};
  \node(22) at (2,-2)[scale=.93]{\6};
  \node(23) at (3,-3)[scale=.93]{\7};

  \node(31) at (1,-4)[scale=.93]{\8};
  \node(33) at (3,-4)[scale=.93]{\9};
  \node(34) at (4,-4)[scale=.93]{\4};
  
  \node(30) at (2,-4.5)[scale=.93]{\frn};
  \node(32) at (1,-5.5)[scale=.93]{\ffn};
  
  \node(41) at (1,-6.5)[scale=.93]{\ten};
  \node(42) at (2,-6.5)[scale=.93]{\lvn}; 
  \node(43) at (3,-6.5)[scale=.93]{\twv};
  \node(44) at (4,-6.5)[scale=.93]{\thn};

 %rows
   \draw[->] (11) -- (12) ;
   \draw[->] (12) -- (13) ;
   \draw[->] (13) -- (14) ;
   
   \draw[->] (21) -- (22) ;
   \draw[<-] (22) -- (23) node[below=-.5pt, midway, scale=0.75]{$\alpha\mkern15mu$};
     
;
  
  \draw[->] (33) -- (34) ;
   
   \draw[->] (41) -- (42) ;
   \draw[->] (42) -- (43) ;
   \draw[->] (43) -- (44) ;

 %columns
   \draw[->] (11) -- (21) ;   
   \draw[->] (21) -- (31) ; 
   \draw[->] (31) -- (32) ;
   \draw[->] (32) -- (41) ;
   
   \draw[->] (12) -- (22) ;   
   \draw[->] (22) -- (30) node[right=1pt, midway, scale=0.75]{$\delta^*\chi^{}_{\>p^{}_1}$}; 
   \draw[->] (30) -- (42) ;
   
    \draw[->] (32) -- (30) ;
    
   \draw[->] (13) -- (23) ;
   \draw[->] (23) -- (33) ;   
   \draw[->] (33) -- (43) node[left=.5pt, midway, scale=0.75]{$\delta^*\<p_2^*\chi^{}_{\<\<f}$};
     
   \draw[double distance=2pt] (14) -- (34) ;   
   \draw[->] (34) -- (44) node[right=1pt, midway, scale=0.75]{$\via\chi^{}_{\<\<f}$}; 

 %labels
   \node at (2.7,-1.95) {\circled1};
   \node at (3.51,-2.3) {\circled2};
   \node at (2.5,-5.4) {\circled4};
   \node at (1.6,-3.5) {\circled3};

  \epic
\]

Commutativity of \circled1  follows from that of the next diagram of natural isomorphisms, 
the commutativity of whose subdiagrams is either obvious or included in 
the pseudofunctoriality of $(-)^*$:

\[
\def\1{$f^*=(p_1\delta)^*f^*$}
\def\2{$(p_2\delta)^*f^*=f^*$}
\def\3{$(fp_1\delta)^*$}
\def\4{$(fp_2\delta)^*$}
\def\5{$\delta^*(fp_1)^*$}
\def\6{$\delta^*(fp_2)^*$}
\def\7{$\delta^*p_1^*f^*$}
\def\8{$\delta^*p_2^*f^*$} 
\bpic[xscale=2.5,yscale=1.25]
 
  \node(11) at (1,-1){\1};
  \node(14) at (4,-1){\2};
  
  \node(22) at (2,-1.9){\3};
  \node(23) at (3,-1.9){\4};

  \node(32) at (2,-3.1){\5};
  \node(33) at (3,-3.1){\6};

  \node(41) at (1,-4){\7};
  \node(44) at (4,-4){\8};

 %rows
   \draw[double distance=2pt] (11) -- (14) ;
   
   \draw[double distance=2pt] (22) -- (23) ;  
   \draw[double distance=2pt] (32) -- (33) ;    
   \draw[double distance=2pt] (41) -- (44) ;

 %columns
   
   \draw[->] (11) -- (41) ;
   
   \draw[->] (22) -- (32) ; 
 
   \draw[->] (23) -- (33) ;
   
   \draw[->] (14) -- (44) ;
   
 %oblique  
 
   \draw[->] (11) -- (22) ;
   \draw[->] (14) -- (23) ;   
   \draw[->] (32) -- (41) ; 
   \draw[->] (33) -- (44) ;

  \epic
\]

Subdiagram  \circled2 expands, naturally, as follows, with $E_1\set\! f^!\<F_1$, $E_2\set \!f^*\<\<F_2$ and~$p\set \<p_2$.

\[
\def\1{$\delta^*\<p^*\<\<E_1\otimes E_2$}
\def\2{$(p\delta)^{\<*}\<E_1\otimes E_2$}
\def\3{$E_1\otimes E_2$}
\def\4{$\delta^*\<p^*\<\<E_1\otimes (p\delta)^{\<*}\<E_2$}
\def\5{$(p\delta)^{\<*}\<E_1\otimes (p\delta)^{\<*}\<E_2$}
\def\6{$(p\delta)^{\<*}(E_1\otimes E_2)$}
\def\7{$\delta^*\<p^*\<\<E_1\otimes \delta^*\<p^*\<\<E_2$}
\def\8{$\delta^*\<(p^*\<\<E_1\otimes p^*\<\<E_2)$}
\def\9{$\delta^*\<p^*(E_1\otimes E_2)$}
 \bpic[xscale=4,yscale=1.5]
 
  \node(11) at (1,-1){\1};
  \node(12) at (2.02,-1){\2};
  \node(13) at (3,-1){\3};
  
  \node(21) at (1,-2){\4};
  \node(22) at (2.02,-2){\5};
  \node(23) at (3,-2){\6};

  \node(31) at (1,-3){\7};
  \node(32) at (2.02,-3){\8};
  \node(33) at (3,-3){\9};

 %rows
   \draw[->] (11) -- (12) ;
   \draw[double distance=2pt] (12) -- (13) ;
      
   \draw[->] (21) -- (22) ;
   \draw[->] (22) -- (23) node[below=1pt, midway, scale=0.75]{$\gamma$};
  
   \draw[->] (31) -- (32) ;
   \draw[->] (32) -- (33) ;

 %columns
   \draw[double distance=2pt] (11) -- (21) ;   
   \draw[->] (21) -- (31) ; 
     
   \draw[double distance=2pt] (12) -- (22) ;   
      
   \draw[double distance=2pt]  (13) -- (23) ;
   \draw[->] (23) -- (33) ;   
     
  %oblique
   \draw[double distance=2pt] (22) -- (13);
   \draw[->] (31) -- (22) ;
   
 %labels
   \node at (2.7,-1.65) {\circled5};
   \node at (2,-2.5) {\circled6};
   
  \epic
\]
Commutativity of \circled6 is given by the dual of the commutative diagram \cite[3.6.7.2]{li} (see proof of \cite[3.6.10]{li}). 
As $p\delta$ is the identity map of~$X,$ the same diagram, with $g=f=\id$, yields that  the isomorphism $\gamma$ is idempotent, whence it is the identity map, so that \circled5 commutes.\va2

Subdiagram \circled3 without $\delta^*$ expands, naturally, to the following, 
with \mbox{$E_1\set f^*\<\<F_1,$} $E_2\set f^*\<\<F_2\>,$ and $p\set p_1\>$:
\[\mkern-5mu
\def\1{$\delta_*\<\<E_1\<\<\otimes\< p^*\<\<E_2$}
\def\2{$\delta_*\delta^!p^!\<E_1\<\<\otimes\< p^*\<\<E_2$}
\def\3{$p^!\<E_1\<\<\otimes\< p^*\<\<E_2$}
\def\4{$\delta_*(\delta^!p^!\<E_1\<\<\otimes\< \delta^*\<p^*\<\<E_2)$}
\def\5{$\delta_*\delta^!(p^!\<E_1\<\<\otimes\< p^*\<\<E_2)$}
\def\6{$\mkern5mu\delta_*\<\big((p\delta)^!\<E_1\<\<\otimes\< (p\delta)^*\<\<E_2\<\big)$}
\def\7{$\delta_*(E_1\<\<\otimes\< \delta^*\<p^*\<\<E_2)$}
\def\8{$\delta_*(E_1\<\<\otimes\< E_2)$}
\def\9{$\delta_*\delta^!p^!(E_1\<\<\otimes\< E_2)$}
\def\ten{$p^!(E_1\<\<\otimes\< E_2)$}
\bpic[, xscale=3.38,yscale=1.45]
 
  \node(11) at (1,-1)[scale=.93]{\1};
  \node(12) at (1.97,-1)[scale=.93]{\2};
  \node(14) at (4,-1)[scale=.93]{\3};
  
  \node(22) at (1.97,-1.9)[scale=.93]{\4};
  \node(23) at (3.11,-1.9)[scale=.93]{\5};

  \node(32) at (1.97,-3)[scale=.93]{\6};
  
  \node(41) at (1,-3)[scale=.93]{\7};
  \node(42) at (1,-4)[scale=.93]{\8}; 
  \node(43) at (3.11,-4)[scale=.93]{\9};
  \node(44) at (4,-4)[scale=.93]{\ten};

 %rows
   \draw[->] (11) -- (12) ;
   \draw[->] (12) -- (14) ;
   
   \draw[->] (22) -- (23) node[below=.5pt, midway, scale=0.75]{$\delta_*\chi^{}_{\delta}$};
     
   \draw[->] (41) -- (42) ;
   \draw[->] (42) -- (43) ;
   \draw[->] (43) -- (44) ;

 %columns
   \draw[->] (11) -- (41) ;   
   
   \draw[->] (12) -- (22) ;
   \draw[->] (22) -- (32) ; 
   \draw[double distance = 2pt] (32) -- (42);
   
   \draw[->] (23) -- (43) node[right=.5pt, midway, scale=0.75]{$\delta_*\delta^!\chi^{}_{p}$};    
   
   \draw[->] (14) -- (44) node[right=.5pt, midway, scale=0.75]{$\chi^{}_{p}$};

 %oblique
   \draw[->] (41) -- (22) ;
   \draw[->] (23) -- (14) ;
 
 %labels
   \node at (2.75,-1.45) {\circled7};
   \node at (2.75,-3.05) {\circled8};

 \epic
\]

Commutativity of subdiagram \circled7 is immediate from the definition of $\chi^{}_{\delta}$
($\delta$ being proper).

In \cite{Nm14} there is a blanket convention that the functors and
natural transformations strictly respect identities, hence
$(p\delta)^*=\id^*=\id=\id^!=(p\delta)^!$. The commutativity of
 \circled8 follows from \cite[Theorem 13.4]{Nm14}.

Commutativity of subdiagram \circled4 is given by \cite[Proposition 13.7]{Nm14}.\va2

Commutativity of all the remaining (unlabeled) subdiagrams is clear.
\end{proof}

\begin{cosa}\label{equidim} Let $f\colon X\to Z$ be a flat $\sE$\kf-map. The map $f$ is \emph{equidimensional of relative dimension}~$d\>$
if for each $x\in X$ that is a generic point of an irreducible component of the fiber $f^{-\<1}\<f(x)$, the transcendence degree of the residue field of the local ring 
$\CO_{\<\<X\<,\>x}$ over that of $\CO_{\<Z\<,\>f(x)}$ is $d$. When $f$ is of finite type, this just means
that every irreducible  component of every fiber has dimension $d$. 

An essentially smooth $\sE$\kf-map of constant relative dimension~$d$ is equi\-dimensional
of relative dimension $d$.

\emph{For any equidimensional such\/ $f\<$, of relative dimension\/~$d$, it holds that\/ $H^e\<f^!\OZ=0$ 
whenever} $\>e<\<\<-d$. Indeed, this assertion is local on
$Z$ and~$X$ (see the remark right after \eqref{bciso}). One gets then from \cite[13.3.1]{EGA4} that $f$ may be assumed to be of the form $\spec B\to\spec A$ where $B$ is a localization of
a module-finite quasi-finite algebra $B_0$ over the polynomial ring $A[T_1,\dots,T_d\>]$. By
Zariski's Main Theorem \cite[8.12.6]{EGA4} the map $\spec B_0\to\spec A[T_1,\dots,T_d\>]$
factors as finite$\,\smallcirc$(open immersion). The isomorphism \eqref{vf} shows that the relative dualizing complex for the map $\spec A[T_1,\dots,T_d\>]\to\spec A$ is concentrated in degree~$-d$, so \eqref{d_f} implies that the assertion holds for any map $\spec B_1\to\spec A$ with $B_1$ finite over $A[T_1,\dots,T_d\>]$, and then
by (i) in~\ref{motive}, it holds for $\spec B\to\spec A$.\looseness=-1

So there is a canonical composite map \looseness=-1
\begin{equation}\label{omegaO to f^!}
\Omega^d_f[d\>]\xto[\eqref{fund1}]{} (H^{-d}\delta^*\<\delta_*\OX)[d\>]\xto[H^{-d}\>\fc{\<f}]{} (H^{-d}f^!\OZ)[d\>]\lto f^!\OZ\>,
\end{equation}
whence a canonical map
\[
c^{}_{\<f}\colon\Omega^d_f[d\>]\otimes f^*\lto f^!\OZ\otimes f^*\underset{\eqref{chi}} \iso \>f^!.
\]

For  essentially smooth $f$ there is then an obvious question---the one that motivated the present paper, and to which the answer is affirmative:

\begin{thm} \label{v and c}
For essentially smooth\/ $\sE$-maps\/ $f\colon X\to Z$ of relative dimen\-sion~$d,$ and\/ $F\in\Dqc(Z),$ the fundamental class map\/ $c^{}_{\<f}(F)$
is inverse to Verdier's  isomorphism\/ $v^{}_{\!f}(F)$.
 \end{thm}

\begin{proof} As in Section~\ref{setup}.
set $Y\set X\times_{\<Z}X$, let $\delta\colon X\to Y$ be the diagonal map, and $p^{}_i\colon Y\to X\ (i=1,2)$ the projections onto the first and second factors, respectively, so that we have the diagram, with fiber square~$\clubsuit\>\>$,\looseness=-1
\begin{equation}\label{club2}
\CD
X @>\delta >> Y @>p^{}_2 >> X \\
@. @V p^{}_1 VV @VVfV \\
@. X @>\lift4.8,\displaystyle\clubsuit,>\lift1.1,f,> Z
\endCD
\end{equation}
Resolving $\delta_*\OX$ locally by a Koszul complex, one sees that $H^e\delta^*\<\<\delta_*\OX=0$ if $e<-d\>$; so 
one has natural maps\looseness=-1
\begin{equation}\label{o to delta}
\omega\set \Omega^d_f[d\>]\xto[\!\eqref{fund1}\<] {}\big(H^{-d}\delta^*\<\<\delta_*\OX\big)[d\>]\lto \delta^*\<\<\delta_*\OX,
\end{equation}
and  $c^{}_{\<f}$ is the composite 
\begin{equation*}
\omega\otimes f^*\xto[\!\eqref{o to delta}\>\otimes\>\id]{} \delta^*\<\delta_*\OX\otimes f^* \xto[\!\fc{\!f}(\OZ)\otimes\>\id\!]{}f^!\OZ\otimes f^* 
\xto[\eqref{chi}]{\chi^{}_{\<\<f}} f^!.
\end{equation*}
In view of Proposition~\ref{v and otimes}, the theorem says then that the composite
\begin{equation*}
\omega\otimes f^*\xto[\!\eqref{o to delta}\>\otimes\>\id]{} \delta^*\<\delta_*\OX\otimes f^* \xto[\!\fc{\!f}(\OZ)\otimes\>\id]{}f^!\OZ\otimes f^* \xto[\!v^{}_{\!f}(\OZ)\otimes\>\id]{} \omega\otimes f^*
\end{equation*}
is the identity map. \va2

\begin{small}\noindent
(Here we implicitly used commutativity of the  diagram of natural isomorphisms
\begin{equation}\label{cfalt}
\CD
f^*\OZ \otimes f^*@>\Iso>> \OX\otimes f^*\\
@V\simeq VV @VV\simeq V\\
f^*(\OX\otimes\id) @>\Iso>>f^*
\endCD
\end{equation}
which commutativity holds since, \emph{mutatis mutandis,} this diagram is 
\emph{dual} \mbox{\cite[3.4.5]{li}} to the commutative subdiagram~\circled2 in the proof of \cite[3.4.7(iii)]{li}.)\va2
\end{small}

It suffices therefore to prove Theorem~\ref{v and c} when $F=\OZ$.\va2

\goodbreak
Let $(-)'$ be the endofunctor $\sHom_X(-,\OX)$ of $\D(X)$. The ``mirror image" of the evaluation map~
$\ev(X,E,\OX)$ (see \eqref{eval}) is the natural composite 
\begin{equation}\label{evalt}
E\otimes E'\iso E'\otimes E\xto[\!\eqref{eval}\!]{} \OX\qquad(E\in\D(X)).
\end{equation}

Now, after unwinding of the definitions involved,
Theorem~\ref{v and c} for $F=\OZ$ states that the border of the next natural diagram  commutes.
(Going~around clockwise from the upper left corner\va{-1.5} to the bottom right one gives $c^{}_{\<f}(\OZ)$, while
going around counterclockwise gives $v_{\!f}(\OZ)^{-1}$.) 

In this diagram,   $\otimes\set\otimes_\sX$, $\psi^0_2\set \psi(\delta, p_2, \OX,\OX)$---the identity map of~$\delta^*\<\delta_*\OX$ (see~\ref{psiO=id}),
and $\psi^{}_2\set \psi(\delta, p_2, \OX,\delta^!p_1^!\OX)$. 
Commutativity of the unlabeled subdiagrams is straightforward to check. The problem is to show\va2 commutativity of~\circled1.\looseness=-1

\[\mkern-5mu
\def\1{$\delta^*\<\delta_*\OX\<\otimes\< \OX$}
\def\2{$\delta^*\<\delta_*\OX$}
\def\3{$\>\>\delta^*\<\delta_*\delta^!p_1^!\OX\mkern-3mu$}
\def\4{$\omega\<\otimes\< \OX$}
\def\5{$\delta^*\<\<\delta_*\OX\<\<\otimes\< \delta^!p_1^!\OX\quad$}
\def\6{$\delta^!p_1^!\OX\<\<\otimes\<\omega$}
\def\7{$\ \omega\<\otimes\< \delta^!\OY\<\otimes\< \delta^*p_1^!\OX$}
\def\8{$\omega\<\otimes\<\omega'\<\otimes\< \delta^*p_1^!\OX$}
\def\9{$\delta^*p_1^!\OX$}
\def\ten{$\delta^!p_2^*f^!\OZ\<\otimes\<\omega\ \ $}
\def\lvn{$\delta^*p_2^*f^!\OZ$}
\def\twv{$\delta^!\OY\<\otimes\< \delta^*p_2^*f^!\OZ\<\otimes\<\omega$}
\def\thn{$\ \omega\<\otimes\< \omega'\<\otimes\< \delta^*p_2^*f^!\OZ$}
\def\frn{$\delta^!\OY\<\otimes\< f^!\OZ\<\otimes\<\omega$}
\def\ffn{$\omega'\<\otimes\< f^!\OZ\<\otimes\<\omega$}
\def\sxn{$f^!\OZ$}
\def\svn{$\OX\otimes\omega$}
\def\egn{$\omega\<\otimes\< \omega'\<\otimes\<f^!\OZ$}
\def\ntn{$\omega\<\otimes\<\delta^!p_1^!\OX$}
\def\twy{$\omega\<\otimes\< \delta^!\OY\<\otimes\< \delta^*p_2^*f^!\OZ\ $}
\bpic[, xscale=3.3,yscale=1.4]
 
  \node(11) at (3.1,-1){\1}; 
  \node(12) at (1,-1){\svn};
  \node(13) at (4,-1){\2};
  \node(21) at (1.9,-1){\4};
    
  \node(14) at (4,-2){\3};
  \node(31) at (1,-2){\6};
  \node(22) at (1.9,-2){\ntn};
  \node(23) at (3.1,-2){\5};

  \node(32) at (1.9,-3){\7};
  \node(33) at (3.1,-3){\8}; 
  \node(34) at (4,-3){\9};
  \node(41) at (1,-3){\ten};  
 
  \node(42) at (1,-4){\twv};
  \node(425) at (1.9,-5){\twy};
  \node(43) at (3.1,-5){\thn};
  \node(44) at (4,-4){\lvn};
  
  \node(51) at (3.1,-6){\egn};
  \node(52) at (1,-6){\frn};
  \node(53) at (2.12,-6){\ffn};
  \node(54) at (4,-6){\sxn};

 %rows
   \draw[double distance=2pt] (11) -- (13) node[above=1pt, midway, scale=0.75]{$\psi^0_2$};
   \draw[->] (13) -- (14) node[right=1pt, midway, scale=0.75]{$\simeq$};
   
   \draw[<-] (31) -- (22)  node[above=.5pt, midway, scale=0.75]{$\Iso$};
   \draw[->] (22) -- (23)  node[above, midway, scale=0.75]{$\eqref{o to delta}\:$};
   \draw[->] (3.5,-2) -- (3.69,-2) node[above, midway, scale=0.75]{$\psi^{}_2$};
   
   \draw[->] (32) -- (33) node[above=1pt, midway, scale=0.75]{$\<\eqref{gamma}$};
   \draw[->] (33) -- (34) node[above=1pt, midway, scale=0.75]{$\<\eqref{evalt}$};
   
   \draw[->] (41) -- (42) node[left=1pt, midway, scale=0.75]{$\simeq$}
                          node[right, midway, scale=0.75]{$\chi^{\<-1}_{\delta}$}; 
   \draw[->] (42) -- (425) node[above=-1pt, midway, scale=0.75]{$\mkern25mu\simeq$};
   \draw[->] (2.45,-5) -- (2.62,-5) node[below=2pt, midway, scale=0.75]{$\<\eqref{gamma}$};
   \draw[->] (43) -- (44) node[below=-3pt, midway, scale=0.75]{$\mkern60mu\eqref{evalt}$};
   
   \draw[->] (52) -- (53) node[above=.5pt, midway, scale=0.75]{$\Iso$}
                                    node[below=1pt, midway, scale=0.75]{\eqref{gamma}};
   \draw[->] (53) -- (51) node[above=.5pt, midway, scale=0.75]{$\Iso$};
   \draw[->] (51) -- (54) node[above=.5pt, midway, scale=0.75]{$\Iso$}
                                    node[below=1pt, midway, scale=0.75]{$\<\eqref{evalt}$};

 %columns
   \draw[<-] (11) -- (21) node[above, midway, scale=0.75]{$\eqref{o to delta}$};  
   \draw[double distance=2pt] (12) -- (21) ; 
   
   \draw[->] (21) -- (22) node[left=1pt, midway, scale=0.75]{$\simeq$}; 
   \draw[->] (22) -- (32) node[left=1pt, midway, scale=0.75]{$\simeq$}
                              node[right, midway, scale=0.75]{$\eqref{chi}$};
    
   \draw[->] (31) -- (41) node[left=1pt, midway, scale=0.75]{$\simeq$}
                          node[right, midway, scale=0.75]{$\eqref{bch}$}; 
   \draw[<-] (31) -- (12) node[left=1pt, midway, scale=0.75]{$\simeq$}; 
   
   \draw[->] (42) -- (52) node[left, midway, scale=0.75]{$\simeq$};
   
   \draw[->] (33) -- (43) node[right=1pt, midway, scale=0.75]{$\simeq$}
                         node[left, midway, scale=0.75]{\eqref{bch}};
   \draw[->] (43) -- (51) node[right=1pt, midway, scale=0.75]{$\simeq$}; 
   \draw[->] (44) -- (54)node[right=1pt, midway, scale=0.75]{$\simeq$};
   
   \draw[->] (14) -- (34) ;
   \draw[->] (34) -- (44) node[right=1pt, midway, scale=0.75]{$\simeq$}
                         node[left, midway, scale=0.75]{\eqref{bch}};
   
 %oblique
   \draw[->] (11) -- (23) node[right=1pt, midway, scale=0.75]{$\simeq$};
 %  \draw[->] (22) -- (41) ;
   \draw[->] (32) -- (425) node[left=1pt, midway, scale=0.75]{$\simeq$}
                                        node[right, midway, scale=0.75]{\eqref{bch}};

 %labels
   \node at (3.13,-2.52) {\circled1};

 \epic
\]

We'll need another description of the map~\eqref{gamma} (Lemma~\ref{altrho} below).

For this, begin by checking  that  the counit map $\delta^*\<\delta_*G\to G\ (G\in\D(X))$ has  right inverses 
$\tau_p(G)$, where $p\colon Y\to X$ is any map such that $p\delta=\id_\sX$ (e.g., $p=p_1$ or $\>p_2$), and $\tau_p(G)$ is the natural composite 
\begin{equation}\label{tau_p}
G\iso \delta^*\<p^*G\iso \delta^*\<p^*p_*\delta_*G\lto \delta^*\delta_*G,
\end{equation}
that is, the composite $G\iso \delta^*\<p^*G\xto{\delta^*\upsilon(G)} \delta^*\delta_*G$,
where
\begin{equation}\label{oops}
\upsilon\colon p^*\to\delta_*
\end{equation}
is the map adjoint to $\id\iso p_*\delta_*$.

Let $\rho_{\>0}$ be the natural composite (of isomorphisms, since $\delta_*\OX$ is perfect)
\[\delta^*\<\delta_{\<*}\delta^!\OY\xto[\eqref{d_f}]{} 
\delta^*\sHom^{}_Y\<(\delta_{\<*}\OX,\OY)\\
\xto[\eqref{f^*Hom}]{}(\delta^*\<\delta_{\<*}\OX)',
 \]
  and set\va{-1}
\begin{equation}\label{defrho}
\rho=\rho_p\set\rho_{\>0}\tau_p(\delta^!\OY)\colon\delta^!\OY\lto (\delta^*\<\delta_{\<*}\OX)'.
\end{equation} 

As noted just before \eqref{fund3}, $H^e(\delta^!\OY)=0$ for any $e\ne d$, whence the counit map $\delta^*\<\<\delta_*\delta^!\OY\to\delta^!\OY$ 
is taken to an isomorphism by the functor~$H^d$. The inverse of this isomorphism is 
$H^d\>\tau_p(\delta^!\OY)$ because $\tau_p(\delta^!\OY)$ is right-inverse to 
$\delta^*\<\delta_*\delta^!\OY\to \delta^!\OY$. 
It follows then from its description via \eqref{homologize} that~\eqref{def-fund2}, 
with $U=X$, $W=Y$ and $n=d$, is the natural composite\va3
\begin{equation}\label{H-tau-rho}
\begin{aligned}
\delta^!\OY&\cong\big(H^d(\delta^!\OY)\big)[-d\>]\\
&\qquad\qquad\xto{\!H^d(\rho)[-d\>]\>}
\big(H^d((\delta^*\<\delta_*\OX)')\big)[-d\>]
\!\iso\!
(\Omega^{\>d}_{\<f\>})'[-d\>],
\end{aligned}
\end{equation}
the last isomorphism arising via  \eqref{H and Hom} (with $n=d\>$) and \eqref{o to delta}.

\begin{sublem} \label{altrho}
For any\/ $\rho=\rho_p$ as in\/ \eqref{defrho}$,$ the map \textup{\eqref{gamma}} factors as
\[
\delta^!\OY\xto{\ \rho\ } (\delta^*\<\delta_*\OX)'\xto{\<\<\via\eqref{o to delta}\>\>}\omega'.
\]
\end{sublem}

\begin{proof}
For any complex $E\in\D(X)$ set $E^{{\sss \ge} d}\set t_{{\sss \ge} d}E$, with $t_{{\sss \ge} d}$
\va{.7}  the truncation functor
(denoted $\tau_{{\sss \ge} d}$ in \cite[\S\:10.1]{li}).\va{.4} The natural map 
$H^d(E)[-d\>]\to E^{{\sss \ge} d}$ is an isomorphism if $H^eE=0$ for all $e>d$, a condition\va{.2}
satisfied when $E=(\delta^*\delta_*\OX)'$ or $\omega'$. 

Accordingly, and in view of  the description of \eqref{gamma} via \eqref{def-fund2} = \eqref{H-tau-rho},
the lemma asserts that the border of the following diagram, whose top row is \eqref{H-tau-rho},
 commutes.
\[
\def\1{$(\delta^!\OY)^{{\sss \ge} d}$}
\def\2{$\big((\delta^*\<\delta_*\OX)'\big)^{{\sss \ge} d}$}
\def\3{$(\Omega^{\>d}_{\<f\>})'[-d\>]$}
\def\4{$\delta^!\OY$}
\def\5{$(\omega')^{{\sss \ge} d}$}
\def\6{$(\delta^*\<\delta_*\OX)'$}
\def\7{$\ \omega'\ $}
\bpic[ xscale=3.3,yscale=1.2]
 
  \node(11) at (1,-1){\1}; 
  \node(12) at (2,-1){\2};
  \node(13) at (3,-1){\3};
  
  \node(21) at (.24,-1){\4};
  \node(22) at (2,-2.1){\5};  
   
  \node(31) at (.24,-3){\6}; 
  \node(33) at (3,-3){\7};  
  
 %rows
   \draw[->] (11) -- (12) ;
   \draw[->] (12) -- (13) ;
       
   \draw[->] (31) -- (33) node[below=1pt, midway, scale=0.75]{$\via\:$\kern-1pt\eqref{o to delta}};

 %columns
   \draw[<-] (11) -- (21) ;
   \draw[->] (21) -- (31) node[left=1pt, midway, scale=0.75]{$\rho$}; 
       
   \draw[->] (12) -- (22) node[left, midway, scale=0.75]{\raisebox{2pt}{$\via$}}
                                    node[right, midway, scale=0.75]{\eqref{o to delta}}; 
   
   \draw[->] (13) -- (33) node[right, midway, scale=0.75]{\eqref{Hom-shift}};

 %oblique
   \draw[->] (31) -- (12) node[above=-4pt, midway, scale=0.75]{$$};
   \draw[double distance=2pt] (22) -- (33) node[right, midway, scale=0.75]{$$};
  
  %labels
   \node at (2.58,-1.87) {\circled2};
  \epic
\]

The unlabeled subdiagrams clearly commute. Subdiagram \circled2 expands naturally as follows,
with $\sHom=\sHom_\sX$:
\[
\def\1{$\big((\delta^*\<\delta_*\OX)'\big)^{{\sss \ge} d}$}
\def\2{$H^{d}\big((\delta^*\<\delta_*\OX)'\big)[-d\>]$}
\def\3{$\big((H^{-d}\delta^*\<\delta_*\OX)'\big)[-d\>]$}
\def\4{$\big((H^{-d}\delta^*\<\delta_*\OX[d\>])'\big)^{{\sss \ge} d}$}
\def\5{$\big((\Omega^{\>d}_{\<f\>})'\big)[-d\>]$}
\def\6{$(\omega')^{{\sss \ge} d}\ $}
\def\7{$\ \omega'$}
\bpic[ xscale=4,yscale=1.6]
 
  \node(11) at (1,-1.2){\1}; 
  \node(12) at (1,-.3){\2};
  \node(13) at (3,-.3){\3};
  
  \node(21) at (1,-2.1){\4};
  \node(23) at (3,-1.2){\5};  
   
  \node(31) at (1,-3){\6}; 
  \node(33) at (3,-3){\7};  
  
 %rows
   \draw[->] (11) -- (12) node[left=1pt, midway, scale=0.75]{$\simeq$} ;
   \draw[->] (12) -- (13) node[above=1pt, midway, scale=0.75]{\eqref{H and Hom}}; ;
       
   \draw[double distance = 2pt] (31) -- (33) ;
   %
 %columns
   \draw[->] (11) -- (21)  ;  
   \draw[->] (21) -- (31) node[left=1pt, midway, scale=0.75]{$\via\:\kern-1pt\eqref{o to delta}$} ; 
       
   \draw[->] (13) -- (23) node[right=1pt, midway, scale=0.75]{$\textup{cf.}\:\kern-1pt\eqref{o to delta}$} ; 
   \draw[->] (23) -- (33) node[right=1pt, midway, scale=0.75]{\eqref{Hom-shift}}; 
 
   \epic
\]

The vertices of this diagram can all be identified with the complex~$G$ that is 
$H^0\sHom(\Omega^{\>d}_{\<f\>},\OX)$ in degree $d$ and 0 elsewhere. When this is done,
all the maps in the diagram are identity maps except for the two labeled~\eqref{H and Hom}
and~\eqref{Hom-shift}, which are both multiplication in $G$ by~$(-1)^{d(d+1)/2}$.  
(See the remarks following equations \eqref{H and Hom}
and~\eqref{Hom-shift}). Hence subdiagram~\circled2 commutes, and Lemma~\ref{altrho} is proved.
\end{proof}

\pagebreak[3]
One has now that subdiagram \circled1 without  $p_1^!\OX$  expands naturally as follows,
with $\chi$ as in \eqref{chi}, $\id$ the identity functor, and
\begin{align*}
\psi^{}_3&\set\psi(\delta,p_2, \OX,\delta^!\OY\otimes_Y \delta^*(-)),\\
\psi^{}_4&\set\psi(\delta,p_2, \delta^!\OY,\delta^*(-)),\\
\psi^{}_7&\set\psi(\delta,p_2, \OX,\delta^!\OY).\\
\end{align*}
\[\mkern-5mu
\def\1{$\omega\<\otimes\<\delta^!$}
\def\2{$\delta^*\<\delta_*\OX\<\otimes\<\delta^!$}
\def\3{$\delta^*\<\delta_*\delta^!$}
\def\4{$\omega\<\otimes\<\delta^!\OY\<\otimes\< \delta^*$}
\def\5{$\delta^*\<\delta_*\OX\<\otimes\<\delta^!\OY\< \otimes\< \delta^*$}
\def\6{$\omega\<\otimes\<(\delta^*\<\delta_*\OX\<)'\<\otimes\< \delta^*$}
\def\7{$\omega\<\otimes\<\omega'\<\<\otimes\< \delta^*$}
\def\8{$\OX\<\otimes\< \delta^*$}
\def\9{$\delta^*$}
\def\ten{$\mkern5mu\delta^*\<\delta_*\delta^!\OY\!\otimes\< \delta^*$}
\def\lvn{$\delta^*\OY\<\otimes\< \delta^*$}
\def\twv{$\delta^*\<\delta_*(\delta^!\OY\<\otimes_Y\< \delta^*)$}
\def\thn{$\delta^*(\delta_*\delta^!\OY\<\otimes_Y\< \id)$}
\def\frn{$\delta^*(\OY\<\otimes_Y\< \id)$}
\def\sxn{$\delta^*\<\delta_*\OX\<\<\otimes\<\<(\delta^*\<\delta_*\OX\<)'\!\otimes\<\< \delta^*$}
\bpic[ xscale=3.5,yscale=1.45]
 
  \node(11) at (1,.1){\1}; 
  \node(12) at (2.3,.1){\2};
  \node(14) at (4,.1){\3};
  
  \node(21) at (1,-1.24){\4};
  \node(22) at (2.3,-1.24){\5};  
  \node(23) at (3.45,-1.24){\twv};
   
  \node(31) at (1,-4.55){\6}; 
  \node(32) at (2.3,-2.75){\ten};
  \node(33) at (3.45,-2.75){\thn};  
 
  \node(41) at (1.67,-3.37){\sxn};
  \node(42) at (2.3,-4.2){\lvn};
  \node(43) at (3.45,-4.2){\frn};
  
  \node(51) at (1,-5.6){\7};
  \node(52) at (2.3,-5.6){\8};
  \node(54) at (4,-5.6){\9} ;
  
 %rows
   \draw[->] (11) -- (12) node[above, midway, scale=0.75]{$\eqref{o to delta}$};
   \draw[->] (12) -- (14) node[above=1pt, midway, scale=0.75]{$\psi^{}_2$}; 
   
   \draw[->] (21) -- (22) node[above, midway, scale=0.75]{$\eqref{o to delta}$};
   \draw[->] (22) -- (23) node[above, midway, scale=0.75]{$\psi^{}_3$};
    
   \draw[->] (32) -- (33) node[below=1pt, midway, scale=0.75]{\eqref{* and otimes}};
   
   \draw[->] (42) -- (43) node[above=1pt, midway, scale=0.75]{\eqref{* and otimes}};
   
   \draw[->] (51) -- (52) node[below=1pt, midway, scale=0.75]{\eqref{evalt}};   
   \draw[->] (52) -- (54) node[above, midway]{$\Iso$};

 %columns
   \draw[->] (11) -- (21) node[left, midway, scale=0.75]{$\chi^{\<-1}_{\delta}$};  
   \draw[->] (21) -- (31) node[left=1pt, midway, scale=0.75]{$\via\>\rho$}; 
   \draw[->] (31) -- (51) node[left=1pt, midway, scale=0.75]{$(\<\<\ref{o to delta}\<)$};
   
   \draw[->] (12) -- (22) node[left, midway, scale=0.75]{$\chi^{\<-1}_{\delta}$}; 
   \draw[->] (22) -- (32) node[right, midway, scale=0.75]{$\psi^{}_7\otimes\id$};
   \draw[->] (32) -- (42) ;
   \draw[->] (42) -- (52) ;
    
   \draw[->] (23) -- (33) node[right, midway, scale=0.75]{\eqref{projn}$^{-\<1}$}; 
   \draw[->] (33) -- (43) ; 
   
   \draw[->] (14) -- (54) ;

 %oblique
   \draw[<-] (23) -- (14) node[above=-4pt, midway, scale=0.75]{$\chi^{\<-1}_{\delta}\mkern35mu$};
   \draw[->] (43) -- (54) node[above=-7pt, midway]{\rotatebox{-48}{$\mkern20mu\Iso$}};
   \draw[->] (32) -- (23) node[right, midway, scale=0.75]{\kern9pt$\psi^{}_{\<4}$};
   \draw[->] (22) -- (41) node[above=1pt, midway, scale=0.75]{$\via\rho\kern20pt$}; 
   \draw[->] (31) -- (41) node[above=-3.5pt, midway, scale=0.75]{$(\<\<\ref{o to delta}\<)\mkern65mu$}; 
   \draw[->] (41) -- (52) node[left, midway, scale=0.75]{\eqref{evalt}\kern3pt};
   
 %labels
   \node at (2.72,-1.77) {\circled3};
   \node at (2.90,-2.43) {\circled4};
   \node at (1.67,-4.95) {\circled6};
   \node at (3.73,-3.37) {\circled5};
   \node at (2.90,-4.95) {\circled8};
   \node at (1.98,-3.8) {\circled7};
 \epic
\]

Commutativity of the unlabeled subdiagrams is clear;

Commutativity of \circled3 is given by Lemma~\ref{transpsi}.

Commutativity of \circled5 (without $\delta^*$) is the definition of $\chi^{}_{\delta}=\chi^{}_{\delta}(\OY,-)$.

Commutativity\! of~\circled6 is given by  that of the first diagram \mbox{in~\cite[3.5.3(h)]{li},} 
with $(A,B,C)\set(\delta^*\<\delta_*\OX,\omega,\OX)$.

For commutativity of  \circled8, cf.~\eqref{cfalt}.

After restoring the term $p_1^!\OX\cong p_2^*f^!\OZ$ omitted above, 
one gets commutativity of \circled4 from Lemma~\ref{psi and iso}, with $E\set\delta^!\OY,$ 
 $F\set f^!\OZ$.

That leaves subdiagram \circled7, which without ``$\,\otimes\: \delta^*\,$" expands naturally as follows,
where $p\set p_2$ and $F''\set\sHom_Y(F,\OY)\ (F\in\D(Y))$: 
\[\mkern-5mu
\def\1{$\delta^*\<\delta_*\OX\<\otimes\<\delta^!\OY$}
\def\2{$\delta^*\<\delta_*\OX\<\otimes\<\delta^*\<\delta_*\delta^!\OY$}
\def\3{$\delta^*\<(\delta_*\OX\<\otimes\<p^*\<\delta^!\OY\<)$}
\def\4{$\delta^*\<\delta_*\OX\<\otimes\<\delta^*(\delta_*\OX)''$}
\def\5{$\delta^*\<\delta_*(\OX\<\otimes\<\delta^*\<p^*\<\delta^!\OY\<)\,$}
\def\6{$\delta^*\<\big(\delta_*\OX\<\otimes\<(\delta_*\OX)''\big)$}
\def\7{$\delta^*\<\delta_*\delta^!\OY$}
\def\8{$\delta^*\<\delta_*\OX\<\<\otimes\<\<(\delta^*\<\delta_*\OX\<)'$}
\def\ten{$\OX\<$}
\def\lvn{$\delta^*\OY$}
\def\thn{$\delta^*(\delta_*\OX\<\<\otimes\<p^*p_*\delta_*\delta^!\OY\<)$}
\def\frn{$\delta^*(\delta_*\OX\<\<\otimes\<\delta_*\delta^!\OY\<)$}
\def\ffn{$\delta^*\<\<\delta_*\delta^*\<p^*\delta^!\OY$}
\bpic[xscale=4.5,yscale=1.5]
 
  \node(11) at (1,-1){\1}; 
  \node(12) at (2.05,-1){\3};
  \node(13) at (3,-1){\5};
  
  \node(21) at (1,-2){\2}; 
  
  \node(22) at (2.05,-2){\frn} ;
  \node(43) at (3,-2){\ffn};

  \node(23) at (3,-3){\7};
  \node(31) at (1,-3){\4}; 
  \node(32) at (2.05,-3){\6};
 
  \node(33) at (3,-4){\lvn};
  \node(41) at (1,-4){\8};
  \node(42) at (2.05,-4){\ten};
   
   %rows
   \draw[->] (11) -- (12) node[above=1pt, midway, scale=0.75]{$\;\eqref{* and otimes}$};
   \draw[->] (12) -- (13) node[above, midway, scale=0.75]{\eqref{projn}\kern2pt};
   
     \draw[->] (31) -- (32) node[below, midway, scale=0.75]{$\eqref{* and otimes}$};

   \draw[->] (41) -- (42) node[below, midway, scale=0.75]{\eqref{evalt}};
   \draw[->] (33) -- (42)  node[above=1pt, midway, scale=0.75]{$\Iso$};

 %columns
   \draw[->] (11) -- (21) node[left, midway, scale=0.75]{$\via\tau_p(\delta^!\mkern-.5mu\OY\<\<)$};  
   \draw[->] (21) -- (31) node[left, midway, scale=0.75]{$\via\eqref{d_f}$}; 
   \draw[->] (31) -- (41) node[left, midway, scale=0.75]{$\via\eqref{f^*Hom}$};
   
   \draw[->] (12) -- (22) node[right, midway, scale=0.75]{$\via\eqref{oops}$} ;
  
   \draw[->] (13) -- (43) node[right, midway, scale=0.75]{$\simeq$} ;   
   \draw[->] (43) -- (23) node[right, midway, scale=0.75]{$\simeq$} ;  
      
 %oblique
    \draw[->] (32) -- (33) node[above=-3pt, midway, scale=0.75]{$\mkern85mu\textup{cf.}\>\eqref{evalt}$} ;
    \draw[->] (23) -- (33) ;
    \draw[->] (21) -- (22) node[below, midway, scale=0.75]{$\eqref{* and otimes}$};
    \draw[<-] (32) -- (22) node[right, midway, scale=0.75]{$\via\eqref{d_f}$};
    
 %labels
   \node at (1.52,-2.6) {\circled7$_2$};
   \node at (1.51,-1.55) {\circled7$_1$};
   \node at (2.55,-2.5) {\circled7$_4$};
   \node at (2.05,-3.5) {\circled7$_3$};

 \epic
\]

Commutativity of subdiagram \circled7$_1$ results from the definition of $\tau_p$.  

Commutativity of subdiagram \circled7$_2$ is clear.

Commutativity of \circled7$_3$ follows from 
\cite[3.5.6(a)]{li} (with $f\set\delta$, $A\set\delta_*\OX$ and $B\set\OY)$.

Subdiagram \circled7$_4\>$, with the initial ``$\>\delta^*\>$" in each term dropped, 
expands naturally as follows:
\[
\def\1{$\delta_*\OX\<\otimes\<\delta_*\delta^!\OY$}
\def\3{$\delta_*\OX\<\otimes\<p^*\delta^!\OY$}
\def\4{$\delta_*(\OX\<\otimes\<\delta^*\<p^*\delta^!\OY)$}
\def\5{$\delta_*\delta^!\OY$}
\def\6{$\delta_*\delta^!\OY\<\otimes\<\delta_*\OX$}
\def\7{$\delta_*(\delta^!\OY\<\otimes\<\OX)$}
\def\8{$\OY$}
\def\9{$(\delta_*\OX)''\otimes\<\delta_*\OX$}
\def\ten{$\delta_*\OX\<\otimes\<(\delta_*\OX)''$}
\def\lvn{$\delta_*\OX\<\otimes\<\delta_*\delta^*\<\<\delta_*\delta^!\OY$}
\def\twv{$\delta_*(\OX\<\otimes\<\delta^*\<\<\delta_*\delta^!\OY)$}
\def\thn{$\delta_*(\OX\<\otimes\<\delta^!\OY)$}
\def\frn{$\delta_*\delta^*\<p^*\delta^!\OY$}
\bpic[xscale=5, yscale=1.5]
 
  \node(11) at (1,-1.8){\3}; 
  \node(12) at (3,-1.8){\4};
  \node(13) at (3,-5){\5};
 
  \node(22) at (1.625,-4){\lvn}; 
  \node(23) at (3,-4){\frn}; 
 
  \node(31) at (1,-3){\1};
  \node(32) at (2.25,-3){\twv} ;
 
  \node(41) at (1,-5){\1};
  \node(42) at (2.25,-5){\thn};

  \node(51) at (1.5,-6.1){\6};
  \node(52) at (2.5,-6.1){\7};
  
  \node(61) at (1,-7.2){\ten};
  \node(62) at (2,-7.2){\9};
  \node(63) at (3,-7.2){\8};

  %rows
  \draw[->] (11) -- (12) node[above, midway, scale=0.75]{\eqref{projn}};
  
  \draw[->] (31) -- (32) node[above, midway, scale=0.75]{\eqref{projn}}; 
     
  \draw[->] (41) -- (42) node[above, midway, scale=0.75]{\eqref{symmon}};
  \draw[->] (42) -- (13) node[above=1pt, midway, scale=0.75]{$\Iso$};
    
  \draw[->] (51) -- (52) node[above, midway, scale=0.75]{\eqref{symmon}};
  
  \draw[->] (61) -- (62) node[above, midway, scale=0.75]{$\Iso$};
  \draw[->] (62) -- (63) node[below, midway, scale=0.75]{\eqref{eval}};

 %columns
   
   \draw[double distance=2pt] (31) -- (41) ;
   \draw[->] (41) -- (61) node[left, midway, scale=0.75]{$\via\eqref{d_f}$};
   
   \draw[->] (22) -- (32) node[below=-2.7pt, midway,scale=0.75]{\kern40pt\eqref{symmon}};
       
   \draw[->] (12) -- (42) node[left=2pt, midway, scale=0.75]{$\simeq$} ;  
   
   \draw[->] (12) -- (23) node[right=1pt, midway, scale=0.75]{$\simeq$} ;
   \draw[->] (23) -- (13) node[right=1pt, midway, scale=0.75]{$\simeq$} ;

   \draw[->] (13) -- (63) ;
   
  %oblique
    \draw[->] (11) -- (31) node[right, midway, scale=0.75]{$\via\eqref{oops}$};
    \draw[->] (31) -- (22) ;
    \draw[<-] (41) -- (22) ;
    \draw[->] (12) -- (32) node[left, midway, scale=0.75]{$\via\eqref{oops}\mkern25mu$};
    \draw[->] (32) -- (42) ;
    \draw[->] (42) -- (52) ;
    \draw[->] (41) -- (51) ;
    \draw[->] (52) -- (13) ;
    \draw[->] (51) -- (62) node[above=-3pt, midway, scale=0.75]{$\zeta(\OX,\OY\<)\<\otimes\<\id\mkern185mu$}
                                  node[above=-3pt, midway, scale=0.75]{$\mkern85mu\textup{see}\,\eqref{sheaf dual}$};

 %labels
   \node at (1.625,-3.45) {\circled7$_{41}$};
   \node at (2.41,-3.7) {\circled7$_{42}$};
   \node at (1.87,-5.6) {\circled7$_{43}$};
   \node at (2.53,-6.67) {\circled7$_{44}$};

 \epic
\]

Here, commutativity of the unlabeled subdiagrams is clear.

The commutativity of subdiagram \circled7$_{41}$ is given by the definition of the 
projection isomorphism~\eqref{projn}.

Commutativity of subdiagram \circled7$_{42}$ (which says, incidentally, that the map~$\upsilon$ in \eqref{oops} is adjoint to the natural isomorphism $\delta^*p^*\iso\id$) follows easily from   
$\tau_p$ in $\eqref{tau_p}$ being
right inverse to the counit map $\delta^*\delta_*G\to G$.

Subdiagram \circled7$_{43}$ commutes because $\delta_*$ is a \emph{symmetric} monoidal functor.
(See \S\ref{adjunctions}.)

 Finally, commutativity of subdiagram \circled7$_{44}$ means that the map $\zeta(\OX,\OY\<)$
is adjoint to the natural composite 
 \begin{equation}\label{adjzeta}
 \delta_*\delta^!\OY\otimes_Y\delta_*\OX\xto{\eqref{symmon}}\delta_*(\delta^!\OY\otimes_\sX\OX)
 = \delta_*\delta^!\OY\lto\OY\>.
 \end{equation}
But for any $F\in\Dqc(Y)$, $\zeta(\OX,F\>)$ is, by definition, right-conjugate to the projection isomorphism 
\begin{equation}\label{conjproj}
\delta_*\delta^*\<G=\delta_*(\delta^*\<G\otimes_\sX\OX) \underset{\eqref{projn}}\osi G\otimes_Y\delta_*\OX,
\end{equation}
that is (see \cite[3.3.5]{li}),  $\zeta(\OX,F)$ is adjoint to the natural composite 
\[
 \delta_*\delta^!F\otimes_Y\delta_*\OX\underset{\eqref{conjproj}\<}{\iso}
  \delta_*\delta^*\<\<\delta_*\delta^!F
\lto \delta_*\delta^!F\lto F,
\]
which for $F\set\OY$ is the same as \eqref{adjzeta}, since for any $A$ and $B\in\D(X)$ 
(e.g., $A=\delta^!\OY$ and $B=\OX$), the following natural diagram
commutes (as follows easily from the definition of \eqref{projn}):
\[
\def\1{$\delta_*A\<\otimes_Y\<\delta_*B$}
\def\2{$\delta_*(A\<\otimes_\sX\<B)$}
\def\3{$\delta_*\delta^*\<\delta_*A\<\otimes_Y\<\delta_*B$}
\def\4{$\delta_*(\delta^*\delta_*\<A\<\otimes_\sX\<B)$}
\bpic[xscale=5, yscale=1.5]
 
  \node(11) at (1,-1){\1}; 
  \node(12) at (2,-1){\2};
  
  \node(21) at (1,-2){\3};  
  \node(22) at (2,-2){\4}; 
  %rows
  \draw[->] (11) -- (12) node[above, midway, scale=0.75]{\eqref{symmon}};
  
  \draw[->] (21) -- (22) node[below, midway, scale=0.75]{\eqref{symmon}}; 
     
 %columns
   
   \draw[->] (.98,-1.2) -- (.98,-1.8) ;
   \draw[<-] (1.02,-1.2) -- (1.02,-1.8) ;
   
   \draw[->] (22) -- (12) ;
          
  %oblique
    \draw[->] (11) -- (22) node[above, midway, scale=0.75]{$\mkern40mu\eqref{projn}$};
   
  \epic
\]
So subdiagram \circled7$_{44}$ commutes. 

Thus, \circled7 commutes, 
whence so does \circled1.
\end{proof}
\end{cosa}

\section{Acknowledgement}
We are grateful to Pramathanath Sastry for a number of comments from which this paper benefited.

\end{document}